\newif\ifdraft
\draftfalse
 
\documentclass[3p, times,review]{elsarticle}

\usepackage[utf8]{inputenc}
\usepackage[english]{babel}
 \usepackage[]{algorithm2e}
\usepackage[utf8]{inputenc}
\usepackage[T1]{fontenc}
\usepackage{array}
\usepackage{amssymb}
\usepackage{amsfonts}
\usepackage{amsmath}
\usepackage{cases}
\newcommand{\red}[1]{\textcolor{black}{#1}}
\newcommand{\blu}[1]{\textcolor{blue}{#1}}
\usepackage{amsthm}
\usepackage{graphicx}
\usepackage{subfig} 
\usepackage{textcomp}
\usepackage{bm}
\usepackage{booktabs}
\usepackage{mathtools}

\newtheorem{theorem}{Theorem}

\newtheorem{assumptions}{Assumptions}
\newtheorem{lemma}[theorem]{Lemma}

\newcommand{\blue}[1]{\textcolor{black}{#1}}
 \newtheorem{corollary}{Corollary}[theorem]

\theoremstyle{definition}

\newcommand{\R}{\mathbb{R}}

\renewcommand{\P}{\mathbb{P}}
\newcommand{\E}{\mathbb{E}}

\newcommand{\e}{\mathrm{e}}
\newcommand{\rmd}{\mathrm{d}}

\newcommand{\Cdop}{C_\mathrm{dop}}
\newcommand{\DSi}{{\text{D}_\text{Si}}}
\newcommand{\Dox}{\text{D}_\text{ox}}

\newcommand{\TSi}{\mathcal{T}^{\text{Si}}_\ell}

\newcommand{\unit}[2]{#1\,\mathrm{#2}}
\DeclareMathOperator{\MSE}{MSE}

\DeclareMathOperator{\EMC}{E_{MC}}
\DeclareMathOperator{\EMLMC}{E_{MLMC}}
\usepackage{color}

\newcommand{\RN}[1]{%
  \textup{\uppercase\expandafter{\romannumeral#1}}%
}
 
\usepackage[figuresright]{rotating}

\begin{document}

\begin{frontmatter}

  \title{An adaptive multilevel Monte Carlo algorithm
    for the stochastic drift-diffusion-Poisson system}

  \author[tuwien,hannover]{Amirreza Khodadadian\corref{cor1}}
  \ead{khodadadian@ifam.uni-hannover.de}
  \author[tuwien]{Maryam Parvizi}
  \ead{maryam.parvizi@tuwien.ac.at}
  \author[tuwien,asu]{Clemens Heitzinger}
  \ead{clemens.heitzinger@tuwien.ac.at}
  \cortext[cor1]{Corresponding author}
  
  \address[tuwien]{Institute of Analysis and Scientific Computing, 
    Vienna University of Technology (TU Wien),
    Wiedner Hauptstraße 8--10, 
    1040 Vienna, Austria}

  \address[hannover]{Institute of Applied Mathematics, Leibniz University of Hannover, Welfengarten 1, 30167
  	Hannover, Germany}

  \address[asu]{School of Mathematical and Statistical Sciences,
    Arizona State University, Tempe, AZ 85287, USA} 


  \begin{abstract}
    We present an adaptive multilevel Monte Carlo algorithm for
    solving the stochastic drift-diffusion-Poisson system with non-zero recombination rate. The
    a-posteriori error is estimated to enable goal-oriented adaptive
    mesh refinement for the spatial dimensions, while the a-priori
    error is estimated to guarantee \red{linear} convergence of the
    $H^1$ error.  In the adaptive mesh refinement, efficient estimation
    of the error indicator gives rise to better error control.    For the stochastic dimensions, we use the multilevel Monte Carlo
    method to solve this system of stochastic partial differential
    equations. Finally, the advantage of the technique developed here
    compared to uniform mesh refinement is discussed using a realistic numerical
    example.
  \end{abstract}
	
 
\end{frontmatter}

 \section{Introduction}
 
 The stochastic drift-diffusion-Poisson (DDP) system is a general model
 for charge transport in random environments. A leading example is the
 field-effect transistor (FET), where the stochastic coefficients can
 describe process variations, noise, and fluctuations in devices as
 diverse as transistors and sensors. Process variations, noise, and
 fluctuations are significantly important especially in devices scaled into
 the deca-nanometer regime, as random effects become more important in
 smaller devices. Among the many sources of noise, random-dopant
 fluctuations (RDF)  \cite{gerrer2012impact,markov2012statistical,lee2018impact} are one of the most important. Random-dopant
 fluctuations stem from the fact that the doping process in the semiconductors leads to a random number and random position of dopants. Therefore, each impurity atom influences the charge transport and the mobilities.
  A schematic
 diagram is shown in Figure~\ref{fig:device}.
 
 In this paper we analyze  a standard adaptive finite element method  of the form
   \textit{SOLVE $\rightarrow$ ESTIMATE $\rightarrow$ MARK $\rightarrow$ REFINE}
\red{\cite{ainsworth2011posteriori}}  to discretize the  stochastic drift-diffusion-Poisson. In order to do the \textit{MARK} step we need to  
compute 
the a  posteriori error estimate for each  element.

A-priori error estimates yield knowledge
 about convergence and stability of the solvers and information on the
 asymptotic behavior of errors for different mesh sizes
 \cite{szabo1991finite}. A-posteriori error estimates make it possible
 to control the mesh on the entire computational domain by using
 adaptive algorithms, i.e., by focusing computational effort on the
 parts of the domain which contribute most to the total error
 \cite{ainsworth1997posteriori}. In adaptive mesh refinement,
 a-posteriori error estimators are used to indicate where the error is
 particularly high, and then more mesh elements are placed in those
 locations. Here we estimate the local error for a coupled system of
 equations. The error estimate indicates which elements should be
 refined or coarsened simultaneously for the Poisson equation and the
 drift-diffusion equations.
 
 \begin{figure}[t!]
   \centering
   \includegraphics[width=0.8\linewidth]{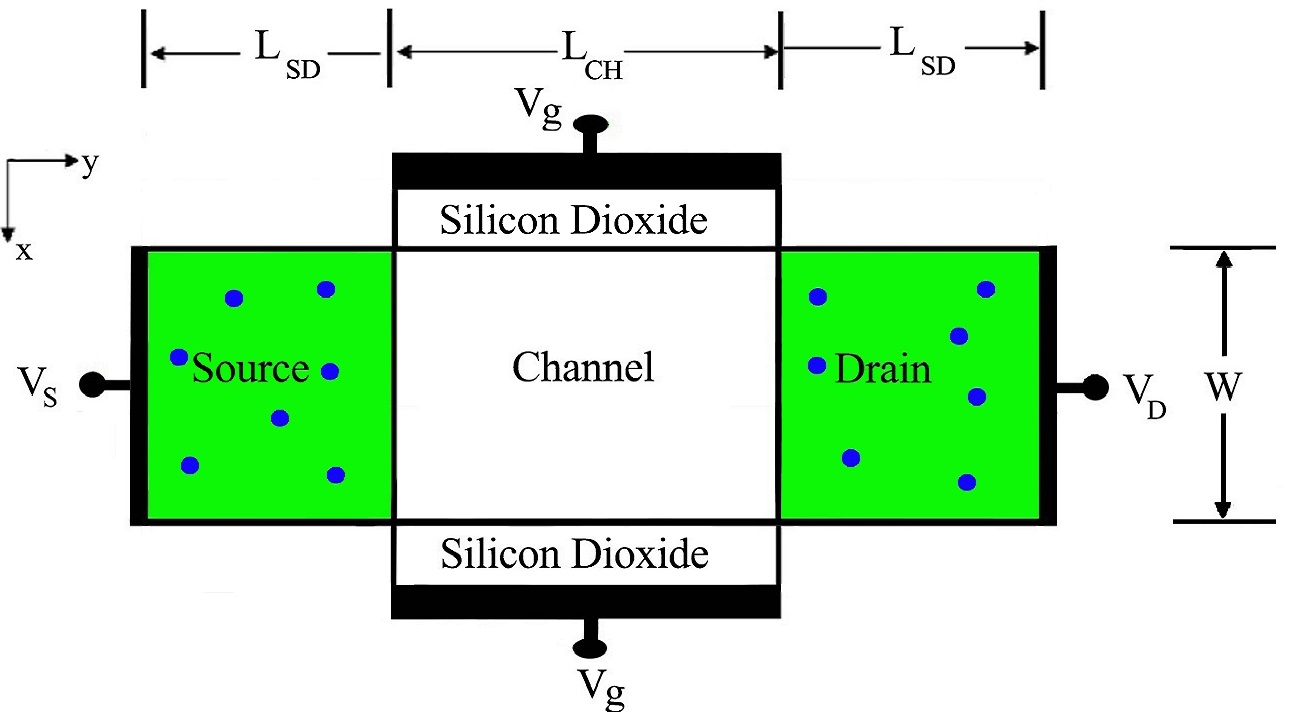}
   \caption{Schematic diagram of a symmetrical double-gate
     metal-oxide-semiconductor field-effect transistor (DG-MOSFET)
     showing its geometrical parameters and the contacts. The blue
     circles in the source and drain regions are the randomly
     distributed impurity atoms.}
   \label{fig:device}
 \end{figure}  	
 		
 The mentioned stochastic problem is computationally very expensive; in
 order to obtain an acceptable error, thousands of simulations are
 necessary. In stochastic PDE, the Monte Carlo (MC) method is one of
 the most popular and straightforward numerical techniques. However,
 the main drawback of the MC method is its well-known convergence rate.
 The multilevel Monte Carlo finite-element (MLMC-FE) method
 \cite{barth2011multi, cliffe2011multilevel,khodadadian2019multilevel} is an efficient numerical
 alternative. In \cite{taghizadeh2017optimal}, we introduced an optimal
 MLMC-FE method to model the random effects in a charge-transport
 model. In the optimal  MLMC-FE method, in each level, we
   determine  mesh sizes $h_\ell$ and the number of samples $M_\ell$
   to minimize  the computational costs such that the total error (i.e., statistical and discretization error) is less than a prescribed error tolerance.
  In \cite{Khodadadian2017three}, the
 efficiency of the method for three-dimensional simulations of various
 nanoscale devices was investigated in detail. 
 Convergence can be
 improved by using a randomized rank-1 lattice rule 
 \cite{giles2009multilevel, Khodadadian2017optimal}.
 	 
 
 The first analysis of a finite-element method, a one-dimensional one,
 for solving the (deterministic) DDP system can be found in
 \cite{cockburn1992convergence}. An extension of the analysis to the
 two-dimensional problem was presented in \cite{chen1995analysis}. In
 \cite{jerome1991finite}, fixed points of finite-element
 discretizations were used to approximate the solutions of the
 steady-state drift-diffusion system, and the convergence rate in the
 energy norm was estimated. In \cite{jerome1995approximation}, the
 optimal convergence rate and its stability were shown. However, in all
 these publications, the
 recombination rate is zero.
 
 In \cite{eigel2014adaptive}, an adaptive stochastic Galerkin
 finite-element method for linear elliptic boundary-value problems was
 presented. The idea of using an adaptive MLMC method for weak
 approximations of solutions of stochastic differential equations was
 explained in \cite{hoel2014implementation}. In
 \cite{eigel2016adaptive}, an adaptive MLMC algorithm was introduced
 for PDEs with stochastic data.
 
 	   
 In \cite{taghizadeh2017optimal}, we showed the effectiveness of using 
 an MLMC method as an alternative to the MC method in the solving the
 stochastic DDP system. There the meshes were refined uniformly.  The
 novelty of the present work is computing an accurate local-error
 estimator based on goal oriented error estimation for the coupled
 system of equations. In order to control the error, only a part
   of the computational geometry which has a large error  must be
 refined and therefore, the method is more computationally effective. Randomness
 in the model problem considered here stems from the random position of
 dopants in a transistor, which affect the error and hence the
 refinement process.  Since MLMC is a variance-reduction method, the
 faster decay in the variance of the MLMC method leads to a reduction
 of the statistical error and therefore the total computational cost.
 Here the effect of adaptive mesh refinement on the reduction of the
 variance is also taken into account. Finally, the developed
   numerical technique can be used for other stochastic problems as well, e.g., \cite{dehghan2016variational,dehghan2017local,dehghan2016meshless,dehghan2015meshless,abbaszadeh2019direct,abbaszadeh2020analysis}.
 	   	
 The remainder of this manuscript is as follows. In Section~2, we give the system of model equations with stochastic coefficients and explain its boundary consitions.
 In Section~3, we formulate a finite element method for the space discretization of the equation.
  In Section~4, we present the error estimates in the
 	finite-element space, namely an a-priori error estimate and an
 	a-posteriori error estimate. In Section~5, we introduce the MLMC
 	finite-element method for the DDP system and define an
 	optimization problem to minimize the total computational
 	cost. Then, we present numerical results for a transistor and
 	quantify the random-dopants effect in Section~6. The adaptive MLMC-FE method is
 	used to approximate the expected value of the solution of the system of
 	equations with random coefficients. Also, the method is compared with
 	MLMC-FE method with uniform mesh refinement. Finally, conclusions are
 	drawn in Section~7.
 	
 \section{The Stochastic Model Problem}\label{model}
 
 The stochastic
 Poisson equation is used generally for the electrostatic potential
 \begin{equation}\label{poisson}
   -\nabla\cdot(A(x,\omega)\nabla V(x,\omega))=\rho(x,\omega),
 \end{equation}
  on the bounded Lipschitz domain $D\subset\mathbb{R}^2$.  In the equation, $V$ indicates the electrostatic potential, $A$ denotes the dielectric constant (permittivity), and $\rho$ is the charge concentration.
  In \eqref{poisson}, $x \in D$ is the spatial variable,
 $\omega=(\omega_1,\omega_2,\ldots,\omega_n)$ is an $n-$dimensional
 random variable defined on the complete probability space $(\Omega,
 \mathbb{A}, \mathbb{P})$ equipped with $\mathbb{A}\subset 2^\Omega$ as the
 $\sigma-$algebra of events, $\mathbb{P}\colon
 \mathbb{A}\rightarrow[0,1]$ as a probability measure, the sample space $\Omega$. The randomness arises from the random distribution of dopant atoms (uniformly distributed) in source and drain areas (shown in Figure \ref{fig:device}). 
 Also, the coefficient $A\colon D \times \Omega \rightarrow \mathbb{R}^{2\times 2}$ is a matrix of real-valued functions, and 
  $\rho$ is a real-valued \red{scalar} function.
 In a semiconductor, the charge concentration is derived by the free
 electron and hole densities~(i.e., $n$ and~$p$) and the doping
 concentration~$C$; the total charge concentration is therefore
 \begin{equation*}\label{eq:charge}
   \rho=q(p-n+C).
 \end{equation*}
 We change the concentrations~$n$ and~$p$ to the
 so-called Slotboom variables~$u$ and~$v$, which are given by
 \begin{subequations}
   \label{e:Slotboom}
   \begin{align}
     n(x,\omega) &=: n_i \e^{V(x,\omega)/U_T}u(x,\omega),\label{eq:np}\\ 
     p(x,\omega) &=: n_i \e^{-V(x,\omega)/U_T}v(x,\omega)\label{eq:np1}.
   \end{align}
 \end{subequations}
 Here, $n_i$ is the intrinsic carrier density of the semiconductor (in
 the numerical examples, a value of $1.5\times 10^{10}\,\text{cm}^{-3}$
 is used for silicon) and $U_T$ indicates the thermal voltage, which is at room temperature is
 about $\unit{26}{mV}$.
 
 A schematic diagram of a sample computational geometry is shown in
 Figure~\ref{fig:device}.  The domain is partitioned into two
 subdomains, i.e.,
 \begin{equation*}
   D = \DSi \cup \Dox. 
 \end{equation*}  
 The first subdomain $\DSi$ consists of silicon, i.e., the channel and
 the source and drain areas, where the drift-diffusion-Poisson system models the charge transport. The gate contact is separated from the
 channel by an insulating silicon dioxide layer $\Dox$.  In $\Dox$,
 there is no charge transport and therefore we only have the Poisson equation.
 
 The boundary $\partial D$ of the domain~$D$ is separated into
 $\partial D_D$ and $\partial D_N$, which  denote the surfaces where
 the Dirichlet and Neumann boundary conditions
 \begin{align}
   V|_{\partial D_D}=V_D \quad \text{and} \quad
   \boldsymbol{n} \cdot \nabla V|_{\partial D_N}=0
 \end{align}
 hold, where $\boldsymbol{n}$ is the outward pointing unit normal on
 the boundary $\partial D_N$. Dirichlet boundary conditions
 $V|_{\partial D_D}$ are applied for the potential at the source,
 drain, and gate contacts, i.e., $V=V_\text{S}$, $V=V_\text{D}$, and
 $V=V_\text{G}$.
 
 Neumann boundary conditions are applied on all other boundaries. On
 the Neumann parts $\partial D_{N}$ of the boundary the currents and
 the electric field are assumed to vanish in the normal direction to
 the surface. This yields the three Neumann boundary conditions
 \begin{align*}
   \partial V (x) \cdot \boldsymbol{n} &= 0\quad  \forall x \in \partial D_N,\\  
   J_n(x) \cdot \boldsymbol{n}         &= 0\quad \forall x \in \partial D_N,\\
   J_p(x) \cdot \boldsymbol{n}         &= 0\quad \forall x \in \partial D_N.
 \end{align*}
 The stochastic DDP system
 \begin{subequations}\label{drift-diffusion}
   \begin{align}
     -\nabla\cdot(A(x,\omega)\nabla V(x,\omega))&=q\left(C(x,\omega)+p(x,\omega)-n(x,\omega)\right),
                                           \label{e:PBE}\\
     \nabla\cdot J_n(x,\omega)&=qR(x,\omega),\\ 
     -\nabla \cdot J_p(x,\omega) &=qR(x,\omega),\\ 
     J_n(x,\omega)&=q(D_n\nabla n(x,\omega)-\mu_n n(x,\omega)\nabla V(x,\omega)),
                    \label{e:DD1}\\ 
     J_p(x,\omega)&=q(-D_p\nabla p(x,\omega)-\mu_p p(x,\omega)\nabla V(x,\omega))
                    \label{e:DD2}
   \end{align}
 \end{subequations}
 is employed to model self-consistent charge transport, where $J_n$ and
 $J_p$ are the current densities,  $\mu_n$ and $\mu_p$ are the
 mobilities, and $D_n$ and $D_p$ are diffusion coefficients, which can
 be calculated by the Einstein relations $D_p = U_T \mu_p$ and $D_n =
 U_T \mu_n$.
 
 Furthermore, $R$ is Shockley-Read-Hall recombination rate, which is
 defined by
 \begin{equation*}
   R_\text{SRH}(n,p):=\frac{n p-n_i^2}{\tau_n(p+n_i)+\tau_p(n+n_i)}, 
 \end{equation*} 
 where $\tau_{n}$ and $\tau_{p}$ are the lifetimes of the free
 carriers (absolutely positive).  For the purposes of the present work, any other
 recombination rate can be used as long as it satisfies modest
 assumptions \cite{taghizadeh2017optimal}.
 
 
 Using the Slotboom variables~$u$ and~$v$ defined in
 \eqref{e:Slotboom}, the DDP system
 \eqref{drift-diffusion} takes the form
 \begin{subequations}\label{uv-drift-diffusion}
   \begin{align}
   -\nabla\cdot(A(x)\nabla V(x,\omega))
   &=q \left(C(x,\omega) -  n_i\left(\e^{V(x,\omega)/U_T}u(x,\omega)
   - \e^{-V(x,\omega)/U_T}v(x,\omega)\right)\right)\\
     U_Tn_i\nabla\cdot(\mu_{n} e^{V/U_T}\nabla u(x,\omega))&=R(x,\omega),\\  
     U_Tn_i\nabla\cdot(\mu_p e^{-V/U_T}\nabla v(x,\omega))&=R(x,\omega) 
   \end{align}
 \end{subequations}
 with the Shockley-Read-Hall recombination rate
 \begin{equation*}
   R_{\text{SRH}}(x,\omega)=n_i\frac{u(x,\omega)v(x,\omega)-1}{\tau_p(e^{V/U_T}u(x,\omega)+1)+\tau_n(e^{-V/U_T}v(x,\omega)+1)}.
 \end{equation*}
 The Dirichlet boundary conditions for the Slotboom variables are
 \begin{equation}\label{dirichlet BCs}
   u(x,\omega)|_{\partial D_{D,\mathrm{Si}}}=u_D(x)\quad \text{and} \quad
   v(x,\omega)|_{\partial D_{D,\mathrm{Si}}}=v_D(x).
 \end{equation}
 Zero Neumann boundary \red{conditions} hold on the Neumann part $\partial D_N$ of the boundary.
 The interface conditions
 \begin{align*}
   V(0+,y,\omega) - V(0-,y,\omega)
   & = 0
   && \text{on } \Gamma,\\
   A(0+)\partial_x V(0+,y,\omega)-A(0-)\partial_x V(0-,y,\omega)
   & = 0
   && \text{on } \Gamma
 \end{align*}
 can be used to model the presence of a layer of charge carriers at the
 surface of a FET after homogenization \cite{Baumgartner2012existence}.
 Here $\Gamma$ is the interface or surface between $\DSi$ and $\Dox$,
 and the notation $0+$ and $0-$ denotes the limits from both sides of
 the interface~$\Gamma$ located at $x=0$. The directions of~$y$ are
 along the interface. The model has been used to model charge transport \cite{khodadadian2015transport} in modern nanoscale devices, see, e.g., \cite{khodadadian2016basis,khodadadian2017design,mirsian2019new,khodadadian2020bayesian}

  Finally, for all $\omega\in \Omega$, we can write the following boundary-value problem (BVP)
  \begin{subequations}\label{modeleqn}
  	\begin{alignat}{2}
  		-\nabla\cdot(A(x)\nabla V(x,\omega))
  		&=q \left(C(x,\omega) -  n_i\left(\e^{V(x,\omega)/U_T}u(x,\omega)
  		- \e^{-V(x,\omega)/U_T}v(x,\omega)\right)\right)
  		~~&&\text{in } \DSi,\\
  		-\nabla\cdot(A(x)\nabla V(x,\omega))
  		& = 0
  		&& \text{in } \Dox,\\
  		{U_T}\nabla \cdot\left( {{\mu _n}{e^{{{V\left( {x,\omega } \right)} \over {{U_T}}}}}\nabla u\left( {x,\omega } \right)} \right) &= {{u\left( {x,\omega } \right)v\left( {x,\omega } \right) - 1} \over {{\tau _p}\left( {{e^{{{V\left( {x,\omega } \right)} \over {{U_T}}}}}u\left( {x,\omega } \right) + 1} \right) + {\tau _n}\left( {{e^{ - {{V\left( {x,\omega } \right)} \over {{U_T}}}}}v\left( {x,\omega } \right) + 1} \right)}}&&\text{in}~\DSi,\\ 	
  		{U_T}\nabla \cdot\left( {{\mu _p}{e^{-{{V\left( {x,\omega } \right)} \over {{U_T}}}}}\nabla v\left( {x,\omega } \right)} \right) &= {{u\left( {x,\omega } \right)v\left( {x,\omega } \right) - 1} \over {{\tau _p}\left( {{e^{{{V\left( {x,\omega } \right)} \over {{U_T}}}}}u\left( {x,\omega } \right) + 1} \right) + {\tau _n}\left( {{e^{ - {{V\left( {x,\omega } \right)} \over {{U_T}}}}}v\left( {x,\omega } \right) + 1} \right)}}&&\text{in}~\DSi,\\      	
  		V(0+,y,\omega) - V(0-,y,\omega)
  		&= 0
  		&& \text{on } \Gamma,\\
  		A(0+)\partial_{x} V(0+,y,\omega)-A(0-)\partial_{x} V(0-,y,\omega)
  		&= 0
  		&& \text{on } \Gamma,\\
  		V(x,\omega) &= V_D(x)
  		&& \text{on } \partial D_D,\\
  		\mathbf{n} \cdot \nabla V(x,\omega) &=0
  		&& \text{on } \partial D_N,\\
  		u(x,\omega) = u_D(x),\quad v(x,\omega) &= v_D(x)
  		&& \text{on } \partial D_{D,\mathrm{Si}},\\
  		\mathbf{n} \cdot \nabla u(x,\omega)=0,\quad
  		\mathbf{n} \cdot \nabla v(x,\omega)&=0
  		&& \text{on }\partial D_{N,\mathrm{Si}}
  	\end{alignat}
  \end{subequations}

In order to state the main result,
 the coefficients and boundary conditions in the boundary-value problem (\ref{modeleqn}) have to
 satisfy the following assumptions.

 \begin{assumptions}\label{assump1}
   \begin{enumerate}
   \item The bounded computational domain $D \subset \R^2$ has a $C^2$ Dirichlet
     boundary $\partial D_D$, the Neumann boundary $\partial D_N$
     consists of $C^2$ segments, and the Lebesgue measure of the
     Dirichlet boundary $\partial D_D$ which is nonzero.  The $C^2$ manifold
     $\Gamma \subset D$ separates the domain~$D$ into two nonempty regions
     $D^+$ and $D^-$; therefore,
     $\mathop{meas}(\Gamma \cap \partial D) = 0$ and
     $\Gamma \cap \partial D \subset \partial D_N$ hold.

  \item The  coefficient   $A(x,\cdot)$ is assumed to be a
 	strongly measurable mapping from $\Omega$ into $L^\infty(\Omega)$. Also, the coefficient~$A$
 	is \red{symmetric},  satisfying the ellipticity condition, and  bounded w.r.t. 
 position $x \in D$
 	and elementary events $\omega \in \Omega$.  Moreover, $\red{A}|_{D^+ \times \Omega} \in C^1(D^+\times
 	\Omega, \R^{2\times2})$ and $\red{A}|_{D^-\times\Omega} \in
 	C^1(D^-\times\Omega, \R^{2\times2})$.
 	\label{assump A}
 	

   \item The doping concentration $\Cdop(x,\omega)$ is bounded above
     and below with the bounds
     \begin{equation*}
       \underline{C}
       := \inf_{x \in D} \Cdop(\cdot,\omega)
       \le \Cdop(\cdot,\omega)\leq \sup_{x \in D} \Cdop(\cdot,\omega)
       =: \overline{C} \quad \forall \omega \in \Omega.
     \end{equation*}
 
 
   \item There is a constant   $ K \ge 1$ which satisfies 
     \begin{equation*}
       \frac{1}{K}\leq u_D(x),v_D(x) \le K
       \quad \forall x \in \partial D_{\mathrm{Si},D}.
     \end{equation*}
 
     
   \item The electron and hole mobilities are
     uniformly bounded functions of $x \in D$ and $\omega \in \Omega$. Therefore, $\forall x \in D$ as well as $\forall\omega\in\Omega$ we have
     \begin{alignat*}{2}
       0<\mu_n^- &\leq \mu_n(x,\omega)\leq \mu_n^+<\infty\\
       0<\mu_p^- &\leq \mu_p(x,\omega)\leq \mu_p^+<\infty,
     \end{alignat*}
     where $\mu_p, \mu_n \in C^1(\DSi\times\Omega,
     \R^{2\times2})$.
     \label{ass7}
     \label{assump mu}
    Moreover, the inclusions $V_D \in H^{1/2}(\partial D )
    	\cap L^{\infty} (\Gamma)$ and $u_D,v_D \in H^{1/2}(\partial
    	\DSi)$ hold.
   \end{enumerate}
 \end{assumptions}
 
 		
\textbf{Remark 1.} \red{For all $\omega\in \Omega$, the matrix
  $A(\cdot,\omega)$ is uniformly elliptic, i.e., there is 
$\red{\hat{A}(\cdot,\omega) \in C^1(D, \R^{2\times2})}$ such that $A(\cdot,\omega)=\hat{A}(\cdot,\omega) *\hat{A}(\cdot,\omega)$.}
	 
 \section{Formulation of the  finite element method }
 	This section is devoted  to present a finite element method for solving \eqref{modeleqn} and separated into two subsections. In the first one, we obtain a  variational formulation and provide some results regarding the existence and uniqueness of the weak solutions. In the second subsection, the finite element discretization based on the mentioned variational formulation is presented. 
 	\subsection{The random variable dependent weak formulation }
  Assume $D \subset \mathbb{R}^2$  satisfies  Assumption \ref{assump1}  . To present the variational formulation of Equation \eqref{modeleqn},  we consider 
 	the Hilbert space  $L^2\left(
 	D\right) $  equipped with the inner product
 	\begin{equation*}
 	\left( u,v \right)_D := \int_D  u(x) v(x) dx
 	\qquad\forall u, v \in L^2(D)
 	\end{equation*}
 	and hence the norm		
 	\begin{equation*}
 	\left\| u \right\|_D = \left( \int_D u(x)^2dx \right)^{\frac{1}{2}}
 	\qquad \forall u \in L^2(D).
 	\end{equation*}
Also, we denote  the Hilbert space equipped with 
%
 the inner product
 $$  {\left( {u,v} \right)_{k,{D}}} := \sum\limits_{0 \le \left| m \right| \le k} {{{\left( {\int\limits_D  {D^mu(x)\,D^mv(x)dx} } \right)}}} $$ and 
 the norm 
  $$
   {\left\| u \right\|_{H^k(D) }}: = {\left( {\sum\limits_{\left| \alpha  \right| \le k} {\left\| {{D^\alpha }u} \right\|_{{L^2}\left( D  \right)}^2} } \right)^{\frac{1}{2}}}
  $$
by $H^k(D)$.
%
 
  Moreover, we define
\begin{align*}
X_1 &:= \{ V \in H^1(D) \mid V|_ {\partial D_D} =V_D \},\\
X_2 &:= \{ u \in H^1(D_{Si}) \mid u|_ {\partial D_{D,Si}} =u_D \},\\
X_3 &:= \{ v \in H^1(D_{Si}) \mid v|_ {\partial D_{D,Si}} =v_D \}.
\end{align*}

 The random variable dependent weak formulation of (\ref{modeleqn}a)--(\ref{modeleqn}d), i.e., the
 Poisson equation on $D$ and the drift-diffusion
 equations for electrons and holes on $\DSi$ for a fixed random variable $\omega\in\Omega$ can be
 written as
 {\fontsize{9.5}{9.5}
 \begin{subequations}\label{weak}
   \begin{alignat}{2}
  \left(\hat{A}\nabla V(\omega),\hat{A}\nabla \varphi_1\right)_{\DSi}&=\left( q{C(\omega)},\varphi_1  \right)_{\DSi} - {\left( {q{n_i}{e^{{V(\omega) \over {{U_T}}}}}u(\omega) - q{n_i}{e^{ - {V(\omega) \over {{U_T}}}}}v(\omega),\varphi_1 } \right)_{\DSi}}
       &\qquad&\forall\varphi_1  \in X_1,\\  
   \left( \hat{A} \nabla V(\omega), \hat{A} \nabla \varphi_1  \right)_{\Dox} &= 0 &&\forall\varphi_1  \in X_1,\\
     - \left( {{U_T}{\mu _n}{e^{{{V\left( {\omega } \right)} \over {{U_T}}}}}\nabla   u\left( {\omega } \right),\nabla \varphi_2 } \right)_{{\DSi}} &= {\left( {{{u(\omega)v(\omega) - 1} \over {\tau_p\left( {{e^{{V(\omega) \over {{U_T}}}}}u(\omega) + 1} \right) + \tau_n\left( {{e^{ - {V(\omega) \over {{U_T}}}}}v(\omega) + 1} \right)}},\varphi_2 } \right)_{{\DSi}}}&&\forall\varphi_2  \in X_2,\\	  		    	           	
     - {\left( {{U_T}{\mu _p}{e^{-{{V\left( {\omega } \right)} \over {{U_T}}}}}\nabla  v\left( {\omega } \right),\nabla \varphi_3 } \right)_{{\DSi}}} &= {\left( {{{u(\omega)v(\omega) - 1} \over {{\tau _p}\left( {{e^{{V(\omega) \over {{U_T}}}}}u(\omega) + 1} \right) + {\tau _n}\left( {{e^{ - {V(\omega) \over {{U_T}}}}}v(\omega) + 1} \right)}},\varphi_3 } \right)_{{\DSi}}}&&\forall\varphi_3  \in X_3,
   \end{alignat}
 \end{subequations}}
  	 We mention  the following lemma to state existence and uniqueness of the continuous solution. 
  \begin{lemma} \cite{taghizadeh2017optimal,Baumgartner2012existence}\label{wellposs}
  		Considering Assumptions~\ref{assump1}, for  $V_D \in H^{1/2}(\partial D)$ and $ u_D, v_D\in H^{1/2}(\partial D_{Si})$, there exists a unique random variable-dependent weak solution
  		\red{\begin{align*}
  	  			V(\cdot,\omega)\in H^1(D\,\backslash\, \Gamma)\,\cap\, L^\infty(D\,\backslash\, \Gamma),\qquad u(\cdot,\omega) \in H^1(D_{\text{Si}})\,\cap\, L^\infty(\text{D}_\text{Si}),
  	  			\qquad v(\cdot,\omega)\in H^1(D_{\text{Si}})\,\cap\, L^\infty(\text{D}_\text{Si})   
  	  			\end{align*}}		
  			for Eq. \eqref{weak}
  \end{lemma}
   \subsection{The random variable dependent finite element  formulation}
 Assume  $D$ is triangulated by a regular (in the sense of Ciarlet), \red{locally} quasi-uniform mesh $\mathcal{T}_\ell$
	with  mesh width $h_\ell$. Also, we assume that the mesh is $\gamma$-shape regular
	in the sense that $\text{diam}(T) \le  \gamma |T|^{1/2} $ for all $T \in \mathcal{T}_\ell$.
  Let $\mathcal{P}^1(T)$ be the space of first degree polynomials on $T \in \mathcal{T}_\ell$. Then, we define
\red{\begin{align}
X_{i,\ell}&:=\left\lbrace u \in X_i : u|_T\in \mathcal{P}_1(T)
            ~~\forall T \in  \mathcal{T}_\ell\right\rbrace  \qquad
            \forall i \in \{1,2,3\}. 
\end{align}}
   Hence, based on the weak formulation
  	(\ref{weak}), we arrive at the
  	random variable dependent  Galerkin discretization  of finding  $(V_\ell (\cdot , \omega),u_\ell(\cdot , \omega), v_\ell(\cdot , \omega)) \in (X_{1,\ell},X_{2,\ell},X_{3,\ell})$ such that 
  	{\fontsize{9.5}{9.5}}
 \begin{subequations}\label{weak1}
   \begin{alignat}{2}
     \left( {\hat{A} \nabla {V_\ell(\omega)}, \hat{A} \nabla \psi_1 } \right)_{\DSi} &= {\left( {q{C(\omega)},\psi_1 } \right)_{\DSi}} - {\left( {q{n_i}{e^{{{{V_\ell}} \over {{U_T}}}}}{u_\ell(\omega)} - q{n_i}{e^{ - {{{V_\ell(\omega)}} \over {{U_T}}}}}{v_\ell(\omega)},\psi_1 } \right)_{\DSi}}&\qquad&\forall\psi_1  \in {X_{1,l}},\\
    {\left( {\hat{A} \nabla {V_\ell(\omega)}, \hat{A} \nabla \psi_1 } \right)}_{\Dox} &= 0&&\forall\psi_1  \in {X_{1,l}},\\
     - {\left( {{U_T}{\mu _n}{e^{{{{V_\ell(\omega)}} \over {{U_T}}}}}\nabla {u_\ell(\omega)},\nabla \psi_2 } \right)_{{\DSi}}}&= {\left( {{{{u_\ell(\omega)}{v_\ell(\omega)} - 1} \over {{\tau _p}\left( {{e^{{{{V_\ell(\omega)}} \over {{U_T}}}}}{u_\ell(\omega)} + 1} \right) + \tau_n\left( {{e^{ - {{{V_\ell(\omega)}} \over {{U_T}}}}}{v_\ell(\omega)} + 1} \right)}},\psi_2 } \right)_{{\DSi}}}&&\forall\psi_2  \in { X_{2,\ell}},\\	
     - {\left( {{U_T}{\mu _p}{e^{-{{{V_\ell(\omega)}} \over {{U_T}}}}}\nabla {v_\ell(\omega)},\nabla \psi_3 } \right)_{{\DSi}}}  &= {\left( {{{{u_\ell(\omega)}{v_\ell(\omega)} - 1} \over {{\tau _p}\left( {{e^{{{{V_\ell(\omega)}} \over {{U_T}}}}}{u_\ell(\omega)} + 1} \right) + {\tau _n}\left( {{e^{ - {{{V_\ell(\omega)}} \over {{U_T}}}}}{v_\ell(\omega)} + 1} \right)}},\psi_3 } \right)_{{\DSi}}}&&\forall\psi_3  \in X_{3,\ell}.
   \end{alignat}
 \end{subequations}
 \red{The existence and uniquness of a FEM formulation of
   \eqref{weak1} is shown in \cite{zlamal1986finite}}.
  \section{ A priori and a posteriori error estimation}
 \label{Section 4}
 In this section, we derive both a-priori and a-posteriori
   error estimates. The main feature of an a-priori error estimate is
   providing knowledge about the asymptotic behavior of the
   discretization error.  Considering an a-posteriori error estimator,
   an adaptive mesh-refinement process consists of the following steps:
   \begin{enumerate}
   	\item 
   Define an initial mesh $\mathcal{T}_\ell$.
   Let $\ell=0$.
   	\item Solve the discrete problem \eqref{weak1} on $\mathcal{T}_0$.
   \item For each element $T \in \mathcal{T}_\ell$, obtain the computed error estimate.
   \item If the computed global error is sufficiently small, then
     stop. Otherwise, determine which elements have to be refined and
     obtain the next mesh $\ell+1$ and return to the second step (2).
   	\end{enumerate}

 For the sake of simplicity, from here we use $h$ instead of $h_\ell$ to denote the mesh width.
 \begin{theorem}[A-priori error estimate]
		\label{theo2}
		Let $(V(\omega),u(\omega),v(\omega))\in \left(H^2(D\setminus
		\Gamma)\cap L^\infty(D\setminus\Gamma)\right)\times\left(
		H^2(\DSi)\cap L^\infty(\DSi)\right)^2$ be the solution of
		(\ref{weak}) and $(V_\ell(\omega), u_\ell(\omega), v_\ell(\omega))\in
		X_{1,\ell}\times X_{2,\ell}\times X_{3,\ell}$ be the solution of (\ref{weak1}). Then, there exists a
		constant $c\in\mathbb{R}^+$ depending on the doping concentration as well as the shape regularity of the mesh such that the inequality
		\begin{align}
		\label{err}
		\left\| {\nabla} (V(\omega)-V_\ell(\omega)) \right\|_{L^2(\text{D})} ^2 + \left\| {\nabla} { (u(\omega)-u_\ell(\omega))} \right\|_{L^2(\DSi)} ^2 + \left\| {\nabla (v(\omega)-v_\ell(\omega))} \right\|_{L^2(\DSi)} ^2 \le c{h^{2}}\left( {\left\| u(\omega) \right\|}_{H^2(\DSi)} ^2 + \left\| v(\omega) \right\|_{H^2(\DSi)} ^2 + \left\| V(\omega) \right\|_{H^2(D)} ^2\right)
		\end{align}
		holds for a fixed random variable $\omega\in\Omega$.
	\end{theorem}

 \begin{proof} For the sake of simplicity, the proof is done in four  steps.\\
 {\em Step 1:} In this step, we provide four equations containing the error terms. 	Let
 	\begin{align*}
 	e_{1,\ell}(\omega) &: = V_\ell(\omega) - \Pi_{1,\ell}V(\omega),\\
 	e_{2,\ell}(\omega) &: = u_\ell(\omega) - \Pi_{2,\ell}u(\omega),\\
 	e_{3,\ell}(\omega) &: = v_\ell(\omega) - \Pi_{2,\ell}v(\omega),
 	\end{align*}
 	where $\Pi_{1,\ell}$ is the $L^2$ projector of  $L^2(D)$ onto $X_{1,\ell}$ and $\Pi_{i,\ell}$  projects  $L^2(D_{Si})$ onto $X_{i,\ell}$ for $i=2,3$.
   To simplify the notations, we drop the random variable $\omega$ from the
   arguments of all functions in the following.
 
   Substituting $\varphi_1=e_{1,\ell}$ into (\ref{weak}a) and
   $\psi_1=e_{1,\ell}$ into (\ref{weak1}a) as well as subtracting these
   two equations, we have the following error equation
   \begin{align}\label{37}
     \nonumber
     {\left( {\hat{A} \nabla {e_{1,\ell}},\hat{A} \nabla  {e_{1,\ell}}}
     \right)_{\DSi}} &= qn_i\left( {\left( {{e^{\frac{{{\Pi_{1,\ell}}V}}{{{U_T}}}}}\Pi_{2,\ell}u -
                       {e^{\frac{{{V_\ell}}}{{{U_T}}}}}{u_\ell}} \right) -
                       \left( {{e^{ - \frac{{\Pi_{1,\ell}V}}{{{U_T}}}}}\Pi_{3,\ell}v - {e^{ - \frac{{{V_\ell}}}{{{U_T}}}}}{v_\ell}}
                       \right),{e_{1,\ell}}} \right)_{\DSi}  \\
     \nonumber
     &+ qn_i\left( {\left( {{e^{-\frac{{\Pi_{1,\ell}V}}{{{U_T}}}}}{\Pi
       _{3,\ell}}v - {e^{-\frac{{{V}}}{{{U_T}}}}}{v}} \right) - \left(
       {{e^{  \frac{{\Pi_{1,\ell}V}}{{{U_T}}}}}\Pi_{2,\ell}u - {e^{ 
       \frac{{{V}}}{{{U_T}}}}}{u}} \right),{e_{1,\ell}}}
       \right)_{\DSi}  \\
     &+{\left( {\hat{A} \nabla  \left( {V - \Pi_{1,\ell}V} \right),\hat{A} \nabla  {e_{1,\ell}}} \right)_{\DSi}}.
   \end{align}
   Substituting $\varphi_1=e_{1,\ell}$ into (\ref{weak}b) and
   $\psi_1=e_{1,\ell}$ into (\ref{weak1}b) as well as subtracting these
   two equations results in
   \begin{align}\label{L}
     {\left( {\hat{A} \nabla {e_{1,\ell}},\hat{A} \nabla {e_{1,\ell}}} \right)_{\Dox}} = {\left( {\hat{A} \nabla \left( {V - \Pi_{1,\ell}V} \right),\hat{A} \nabla  {e_{1,\ell}}} \right)_{\Dox}}.
   \end{align}
   Furthermore, substituting $\varphi_2=e_{2,\ell}$ into (\ref{weak}c) and
   $\psi_2=e_{2,\ell}$ into (\ref{weak1}c) and subtracting these two
   equations leads to
      {\fontsize{9.25}{9.25}
   \begin{align*}
     {\left( U_T\mu_n{{e^{\frac{V_\ell}{U_T}}}\nabla u_\ell - {e^{\frac{{{\Pi_{1,\ell} V}}}{{{U_T}}}}}\nabla {\Pi_{2,\ell} u},\nabla {e_{2,\ell}}} \right)_{{\DSi}}} =& {\left( {\frac{{1-u_\ell v_\ell}}{{{\tau_p}\left( {{e^{\frac{V_\ell}{U_T}}}u_\ell + 1} \right) + {\tau _n}\left( {{e^{-\frac{V_\ell}{U_T}}v_\ell+1}} \right)}} - \frac{1-\Pi_{2,\ell} u~\Pi_{3,\ell}v}{{{\tau _p}\left( {{e^{\frac{\Pi_{1,\ell}V}{U_T}}}{\Pi_{2,\ell} u} + 1} \right) + {\tau _n}\left( {{e^{ - \frac{\Pi_{1,\ell}V}{U_T}}}\Pi_{3,\ell}v + 1} \right)}},{e_{2,\ell}}} \right)_{{\DSi}}}\\                                                                 &+{\left( {\frac{{uv-1 }}{{{\tau _p}\left( {{e^{\frac{V}{{{U_T}}}}}u + 1} \right) + {\tau _n}\left( {{e^{ - \frac{V}{{{U_T}}}}}v}+1 \right)}} - \frac{{ \Pi_{2,h}u~\Pi_{3,\ell}v-1}}{{{\tau _p}\left( {{e^{\frac{{\Pi_{1,\ell}V}}{{{U_T}}}}}\Pi_{2,\ell}u + 1} \right) + {\tau_n}\left( {{e^{ - \frac{{\Pi_{1,\ell}V}}{{{U_T}}}}}\Pi_{3,\ell}v + 1} \right)}},{e_{2,\ell}}} \right)_{{\DSi}}}\\
&+\left(e^{\frac{V}{U_T}}\nabla u-e^{\frac{\Pi_{1,\ell} V}{U_T}} \nabla \Pi_{2,\ell} u,e_{2,h} \right)_{\DSi}.                                      
   \end{align*}}
   Again, substituting $\varphi_3=e_{3,\ell}$ into (\ref{weak}d) and
   $\psi_3=e_{3,\ell}$ into (\ref{weak1}d) and subtracting these two
   equations yields
    {\fontsize{9.25}{9.25}
   \begin{align*} 
     {\left( {U_T}{\mu _p}{{e^{-\frac{V_\ell}{U_T}}}\nabla  v_\ell - {e^{-\frac{\Pi_{1,\ell} V}{U_T}}}\nabla \Pi_{3,\ell} v,\nabla
     {e_{3,\ell}}} \right)_{{\DSi}}} &= {\left( {\frac{{1 - u_\ell v_\ell}}{{{\tau _p}\left({{e^{\frac{V_\ell}{U_T}}} u_\ell + 1}
           \right) + {\tau_n}\left( {{e^{-\frac{V_\ell}{{U_T}}}}v_\ell + 1} \right)}} - \frac{1 -\Pi_{2,\ell} u~\Pi_{3,\ell} v}{{{\tau _p}\left(
                                    {{e^{\frac{{{\Pi_{1,\ell} V}}}{{{U_T}}}}}{\Pi_{2,\ell}u}
                                    + 1} \right) + {\tau _n}\left(
                                    {{e^{ -
                                    \frac{{{\Pi_{1,\ell}V}}}{{{U_T}}}}}{\Pi_{3,\ell}v} +
                                    1} \right)}},{e_{3,\ell}}}
                                    \right)_{{\DSi}}}  \\
                                  &+{\left( {\frac{{ uv-1}}{{{\tau _p}\left( {{e^{\frac{V}{{{U_T}}}}}u + 1} \right) + {\tau_n}\left( {{e^{ - \frac{V}{{{U_T}}}}}v + 1} \right)}} - \frac{{  \Pi_{2,\ell}u~\Pi_{3,h}v - 1}}{{{\tau _p}\left( {{e^{\frac{{\Pi_{1,\ell}V}}{{{U_T}}}}}\Pi_{2,\ell}u + 1} \right) + {\tau_n}\left( {{e^{ - \frac{{\Pi_{1,\ell}V}}{{{U_T}}}}}\Pi_{3,\ell}v + 1} \right)}},{e_{3,\ell}}} \right)_{{\DSi}}}\\ 
    &+\left(e^{-\frac{V}{U_T}}\nabla v-e^{-\frac{\Pi_{1,\ell} V}{U_T}} \nabla \Pi_{3,\ell} v ,e_{2,\ell}\right)_{\DSi}.  
   \end{align*}}
{\em Step 2:} In this step, we simplify some  terms of the error  equations from the previous step. For this purpose, we 
 consider the following notations
\begin{align}
\label{simp1}
L_1(u,v,w)&:=e^{\frac{u}{U_T}}(v-w),\\
L_2(u,v,w,z)&:=\left(e^{\frac{u}{U_T}}-e^{\frac{v}{U_T}}\right)(w-z),\\
\label{simp2}
L_3(u,v,w)&:=\left(e^{\frac{u}{U_T}}-e^{\frac{v}{U_T}}\right)w.
\end{align}
Then, using the  notations from  \eqref{simp1}-\eqref{simp2}, we can write
 \begin{align}\label{24}\nonumber
e^{\frac{V_\ell}{U_T}}\nabla u_\ell-e^{\frac{\Pi_{1,\ell} V}{U_T}}\nabla \Pi_{2,\ell} u=& L_1(V,\nabla u_\ell,\nabla \Pi_{2,\ell} u)+L_2(V_\ell,\Pi_{1,\ell}V,\nabla u_\ell,\nabla \Pi_{2,\ell} u)\\
&+L_1(\Pi_{1,\ell} V,\nabla u_\ell,\nabla \Pi_{2,\ell} u) +L_3(\Pi_{1,\ell}V,V_\ell,\nabla \Pi_{2,\ell} u). 
\end{align}
and
 \begin{align}\label{25}\nonumber
 e^{\frac{-V_\ell}{U_T}}\nabla v_\ell-e^{\frac{-\Pi_{1,\ell} V}{U_T}}\nabla \Pi_{3,\ell} v=& L_1(-V,\nabla v_\ell,\nabla \Pi_{3,\ell} v)+L_2(-V_\ell,-\Pi_{1,\ell}V,\nabla v_\ell,\nabla \Pi_{3,\ell} v)\\
 &+L_1(-\Pi_{1,\ell} V,\nabla v_\ell,\nabla \Pi_{3,\ell} v) +L_3(-\Pi_{1,\ell}V,-V_\ell,\nabla \Pi_{3,\ell} v). 
 \end{align}
Also, It can be easily seen 
\begin{align}\label{26}
e^{\frac{V}{U_T}}u-e^{\frac{\Pi_{1,\ell}V}{U_T}}\Pi_{2,\ell}u=e^{\frac{\Pi_{1,\ell}V}{U_T}}(u-\Pi_{2,\ell} u)+(e^{\frac{V}{U_T}}-e^{\frac{\Pi_{1,\ell} V}{U_T}})u.
\end{align}
and 
\begin{align}\label{26a}
e^{\frac{-V}{U_T}}v-e^{\frac{-\Pi_{1,\ell}V}{U_T}}\Pi_{3,\ell}v=e^{\frac{-\Pi_{1,\ell}V}{U_T}}(v-\Pi_{3,\ell} v)+(e^{\frac{-V}{U_T}}-e^{\frac{-\Pi_{1,\ell} V}{U_T}})v.
\end{align}
Moreover,  for the sake of simplicity   we then define 

   \begin{equation*}
     {F}\left( {u,v,w} \right): = \frac{{1 - vw }}{{{\tau _p}\left( {{e^{\frac{u}{{{U_T}}}}}v + 1} \right) + {\tau_n}\left( {{e^{ - \frac{u}{{{U_T}}}}}w + 1} \right)}}.
   \end{equation*}
  Substituting (\ref{24})--(\ref{26}) into (\ref{37})--(\ref{L}),  together with the above notation
   we infer that
   \begin{align}\label{48}
 \nonumber\left\| {\hat{A} \nabla {e_{1,\ell}}} \right\|_{L^2(\text{D})} ^2 &+ \left({U_T}{\mu _n} L_1(V,\nabla u_\ell,\nabla \Pi_{2,\ell} u ),\nabla  e_{2,\ell} \right)_{\DSi} + \left(U_T\mu _p L_1(V,\nabla v_\ell,\nabla \Pi_{3,\ell} v) ,\nabla  e_{3,\ell} \right)_{\DSi}\\\nonumber&
 +\left(\mu _T \mu_n L_3(\Pi_{1,\ell}V,\nabla u_\ell,\nabla \Pi_{2,\ell} u),\nabla e_{2,h}\right)_{\DSi}  +  \left(\mu _T \mu _p L_3(-\Pi_{1,\ell}V,\nabla v_\ell,\nabla \Pi_{3,\ell} v),\nabla e_{3,\ell}\right)_{\DSi} 
 \\\nonumber
 =&
qn_i\left( {\left( {{e^{\frac{{{\Pi_{1,\ell}}V}}{{{U_T}}}}}\Pi_{2,\ell}u -
		{e^{\frac{{{V_\ell}}}{{{U_T}}}}}{u_\ell}} \right) -
	\left( {{e^{ - \frac{{\Pi_{1,\ell}V}}{{{U_T}}}}}\Pi_{3,\ell}v - {e^{ - \frac{{{V_\ell}}}{{{U_T}}}}}{v_\ell}}
	\right),{e_{1,\ell}}} \right)_{\DSi}  \\
\nonumber
&+ qn_i\left( {\left( {{e^{-\frac{{\Pi_{1,\ell}V}}{{{U_T}}}}}{\Pi
			_{3,\ell}}v - {e^{-\frac{{{V}}}{{{U_T}}}}}{v}} \right) - \left(
	{{e^{  \frac{{\Pi_{1,\ell}V}}{{{U_T}}}}}\Pi_{2,\ell}u - {e^{ 
				\frac{{{V}}}{{{U_T}}}}}{u}} \right),{e_{1,\ell}}}
\right)_{\DSi}  \\\nonumber
&+\left(\hat{A} \nabla (V-\Pi_{1,\ell}V),\hat{A} \nabla e_{1,\ell} \right)_{\DSi}+\left(F(V_\ell,u_\ell,v_\ell) -F(\Pi_{1,\ell} V,\Pi_{2,\ell} u,\Pi_{2,\ell}v),e_{2,\ell} \right)_{\DSi} \\\nonumber
 &+\left(F(V,u,v) -F(\Pi_{1,\ell} V,\Pi_{2,\ell} u,\Pi_{2,\ell}v),e_{2,\ell} \right)_{\DSi}+ \left(F(V_\ell,u_\ell,v_\ell) -F(\Pi_{1,\ell} V,\Pi_{2,\ell} u,\Pi_{2,\ell}v),e_{3,\ell} \right)_{\DSi} \\ 
	 & \nonumber+\left(F(V,u,v) -F(\Pi_{1,\ell} V,\Pi_{2,\ell} u,\Pi_{2,\ell}v),e_{3,\ell} \right)_{\DSi}+\left( e^{\frac{\Pi_{1,\ell} V}{U_T}} \nabla \Pi_{2,\ell} u -e^{\frac{V}{U_T}}\nabla u ,\nabla e_{2,\ell} \right)_{\DSi} \\ 
 &\nonumber +\left( e^{\frac{-\Pi_{1,\ell} V}{U_T}} \nabla \Pi_{2,\ell} v -e^{-\frac{V}{U_T}}\nabla v ,\nabla e_{3,\ell} \right)_{\DSi}-\left\lbrace \left(\mu _T \mu_n L_2(V_\ell,\Pi_{1,\ell}V,\nabla u_\ell,\nabla \Pi_{2,\ell} u)+L_1(\Pi_{1,\ell} V,\nabla u_\ell,\nabla \Pi_{2,\ell} u),\nabla e_{2,\ell}\right) _{\DSi}\right\rbrace 
 \\&-\left\lbrace \left(\mu _T \mu _p   L_2(-V_\ell,-\Pi_{1,\ell}V,\nabla v_\ell,\nabla \Pi_{3,\ell} v)+L_1(-\Pi_{1,\ell} V,\nabla v_\ell,\nabla \Pi_{3,\ell} v),\nabla e_{3,\ell}\right)_{\DSi} \right\rbrace .
  \end{align}

 {\em Step 3:} In this step, we strive to find an upper bound for the terms given in the right hand side of \eqref{48}.\\
We denote the functional  $M_1: X_{1,\ell}\times X_{1,\ell}\times X_{3,\ell}\rightarrow \mathbb{R}$ by
   \begin{equation*}
     M_1(\Pi_{1,\ell}V,V_h) (e_{3,\ell}):=\left(F(\Pi_{1,\ell} V,\Pi_{2,\ell} u,\Pi_{3,\ell} v)- F( V_\ell,u_\ell,v_\ell),e_{3,\ell}\right)_{\DSi}.
   \end{equation*}    
Moreover, we define
   \begin{align*}
   \theta_1:=\theta\, V_\ell+(1-\theta)\,\Pi_{1,\ell}V=\theta\, e_{1,\ell}+e_{1,\ell}-V_\ell=(1+\theta)\,e_{1,\ell}-V_\ell,
   \end{align*} 
    where $0<\theta<1$.  Using the Taylor expansion for the functional $M_1$ over $\Pi_{1,\ell} V$ leads to		

    \begin{equation}\label{expansion} 
         M_1(\Pi_{1,\ell}V,V_\ell) (e_{3,\ell})=
    M_1(\Pi_{1,\ell}V,\Pi_{1,\ell}V)(e_{3,\ell})+\int_{0}^{1}M_1^\prime \left(\Pi_{1,\ell} V,\theta_1 \right)(e_{3,\ell})~\text{d}\,\theta.
      \end{equation} 
 Since the  $F$ is differentiable, considering  the directional derivative of $M_1$ in the  direction of projection of $V$  leads to
   \begin{align} \label{directional}
        \nonumber
          M_1^\prime(\Pi_{1,\ell}V,\theta_1)(e_{3,\ell})&=\nonumber
        \lim\limits_{r\to
                                          0}\frac{M_1(\Pi_{1,\ell}V,\,\theta_1+r\,\Pi_{1,\ell}V)(e_{3,\ell})-M_1(\Pi_{1,\ell}V,\theta_1)(e_{3,\ell})}{r}\\\nonumber
    &\leq \lim\limits_{r\to 0}\frac{\left(F(\Pi V_\ell,\Pi_{2,\ell} u,\Pi_{3,\ell}               v)-F(\theta_1+r\Pi_{1,\ell} V,u_\ell,v_\ell),e_{3,\ell}\right)_{\DSi}}{r}\\\nonumber
                                        &- \lim\limits_{r\to
                                          0}\frac{\left(F(\Pi_{1,\ell} V,\Pi_{2,\ell} u,\Pi_{3,\ell} v)-F(\theta_1,u_\ell,v_\ell),e_{3,\ell}\right)_{\DSi}}{r}\\   
 &\leq c(\theta_1)\,|(e_{1,\ell},e_{3,\ell})_{\DSi}|.
      \end{align}
      Here, we should note that  $F$ is Lipschitz continuous (see Eq. (3.13) of \cite{zlamal1986finite})      
     and the constant	$c(\theta_1)$ is independent of the mesh width. Considering
   \eqref{expansion} and \eqref{directional}, we conclude that
   \begin{equation*}
 \left| M_1(\Pi_{1,\ell}V,V_\ell)(e_{3,\ell})\,\right|\leq\left|\left(F(\Pi_{1,\ell}V,\Pi_{2,\ell}u,\Pi_{2,\ell}
         v)-F(\Pi_{1,\ell} V,u_\ell,v_\ell),e_{3,\ell}\right)_{\DSi}\right|
     +\left|\int_{0}^{1} M_1^\prime(\Pi_{1,\ell}V,\theta_1)(e_{3,\ell})\, \text{d}\theta~\right|.
   \end{equation*}
   Therefore, Eq. \eqref{directional} and the triangle inequality give
   rise to
   \begin{align}\label{M1}
     \left|M_1(\Pi_{1,\ell} V,V_\ell) (e_{3,\ell})\right|\lesssim\left|(F(\Pi_{1,\ell}V,\Pi_{2,\ell}u,\Pi_{1,\ell}V)-F(\Pi_{1,\ell}V,u_\ell,v_\ell),e_{3,\ell})_{\DSi} \right|+\left|(e_{2,\ell},e_{3,\ell})_{\DSi}\right|.
   \end{align}	
   In the next step, we define the functional
   \begin{equation*}
    M_2(\Pi_{2,\ell}u,u_\ell)  (e_{3,\ell}):=\left(F(\Pi_{1,\ell}V,\Pi_{2,\ell}u,\Pi_{3,\ell}v)-F(\Pi_{1,\ell}V,u_\ell,v_\ell),e_{3,\ell}\right)_{\DSi}.
   \end{equation*}
   Similarly, we define the variable
\begin{align*}
\theta_2:=(1+\theta)e_{1,\ell}-u_\ell,
\end{align*}   
where $0<\theta<1$. Taylor expansion of $M_2$ over $\Pi_{2,\ell} u$ shows that
   \begin{equation}\label{MM2}
     M_2(\Pi_{2,\ell}u,u_\ell)(e_{3,\ell})= M_2(\Pi_{2,\ell} u,\Pi_{2,\ell}u)\,(e_{3,\ell})+\int_{0}^{1}M_2^\prime \left(\Pi_{2,\ell} u,\theta_2 \right)(e_{3,\ell})~\text{d}\,\theta.
   \end{equation}
    Again, since the function $f$ is differentiable, the directional derivative of $M_2$ is given by
 \begin{align} \label{directional1}
         \nonumber
           M_2^\prime(\Pi_{2,\ell}u,\theta_2)(e_{3,\ell})&=\nonumber
         \lim\limits_{r\to
                                           0}\frac{M_2(\Pi_{2,\ell}u,\,\theta_2+r\,\Pi_{2,\ell}u)(e_{3,\ell})-M_2(\Pi_{2,\ell}u,\theta_2)(e_{3,\ell})}{r}\\\nonumber
     &= \lim\limits_{r\to 0}\frac{\left(F(\Pi V_\ell,\Pi_{2,\ell} u,\Pi_{3,\ell}
                                            v)-F(V_\ell,\theta_2+r\Pi_{2,\ell}u,v_\ell),e_{3,\ell}\right)_{\DSi}}{r}\\\nonumber
                                         &- \lim\limits_{r\to
                                           0}\frac{\left(F(\Pi_{1,\ell} V,\Pi_{2,\ell} u,\Pi_{3,\ell} v)-F(V_\ell,\theta_2,v_\ell),e_{3,\ell}\right)_{\DSi}}{r}\\      
                                         &\leq c(\theta_2)\,|(e_{1,\ell},e_{3,\ell})_{\DSi}|,
       \end{align}
where again the constant $c(\theta_2)$ is independent of~$h$. Therefore, we have
   \begin{equation*}
\left|  M_2(\Pi_{2,\ell}u,u_\ell)(e_{3,\ell})\,\right|\leq\left|\left(F(\Pi_{1,\ell}V,\Pi_{2,\ell}u,\Pi_{3,\ell}
         v)-F(\Pi_{1,\ell} V,u_\ell,v_\ell),e_{3,\ell}\right)_{\DSi}\right|
     +\left|\int_{0}^{1} M_2^\prime(\Pi_{2,\ell}u,\theta_2)(e_{3,\ell})\, \text{d}\theta~\right|,
   \end{equation*}
and applying the triangle inequality to \eqref{MM2} results in
  \begin{equation}\label{inequ11} 
     \left|M_2(\Pi_{2,\ell} u,u_\ell)(e_{3,\ell})\right|\lesssim \left|\left(F(\Pi_{1,\ell}V,\Pi_{2,\ell}u,\Pi_{3,\ell}
              v)-F(\Pi_{1,\ell} V,u_\ell,v_\ell),e_{3,\ell}\right)_{\DSi}\right|+\left|(e_{1,\ell},e_{3,\ell})_{\DSi}\right|.
   \end{equation}
Now using the triangle inequality and applying the the approximation and the stability properties of   $\Pi_{1,\ell}$, $\Pi_{2,\ell}$, and $\Pi_{3,\ell}$, we have
\begin{align}\nonumber
|\left( L_2(V_\ell,\Pi_{1,\ell}V,\nabla u_\ell,\nabla \Pi_{2,\ell} u)+L_1(\Pi_{1,\ell} V,\nabla u_\ell,\nabla \Pi_{2,\ell} u),\nabla e_{2,\ell}\right) _{\DSi}|&\leq\left( \|e_{1,\ell}\|_{L^2(\DSi)} \|\nabla  e_{2,\ell} \|_{L^2(\DSi)}+ \|e_{1,\ell}\|_{L^2(\DSi)} \|\nabla  e_{3,\ell} \|_{L^2(\DSi)}  \right)\|\nabla e_{2,\ell} \|_{L^2(\DSi)}
\end{align}
as well as
 {\fontsize{9.5}{9.5}
 \begin{align}\nonumber
	|\left( L_2(-V_\ell,-\Pi_{1,\ell}V,\nabla v_\ell,\nabla \Pi_{3,\ell} v)+L_1(-\Pi_{1,\ell} V,\nabla v_\ell,\nabla \Pi_{3,\ell} v),\nabla e_{3,\ell}\right) _{\DSi}|&\leq\left( \|e_{1,\ell}\|_{L^2(\DSi)} \| \nabla  e_{3,\ell} \|_{L^2(\DSi)}+ \|e_{1,\ell}\|_{L^2(\DSi)} \|\nabla  e_{3,\ell} \|_{L^2(\DSi)}  \right)\|\nabla  e_{3,\ell} \|_{L^2(\DSi)}
	\end{align}}
and
 {\fontsize{9}{9}
 	\begin{align}\nonumber
 		|\left(L_3(\Pi_{1,\ell}V,V_\ell,\nabla \Pi_{2,\ell} u),\nabla e_{2,h}\right)_{\DSi}   +  \left(L_3(-\Pi_{1,\ell}V,-V_\ell,\nabla \Pi_{3,\ell} v),\nabla e_{3,\ell}\right)_{\DSi}  |&\lesssim \left(\|\nabla e_{2,h}\|_{L^2(\DSi)} \|\nabla \Pi_{2,\ell} u\|_{L^2(\DSi)}\right.\\ \nonumber
 		& \left. +\|\nabla e_{3,h} \|_{L^2(\DSi)} \|\nabla\Pi_{2,\ell}v \|_{L^2(\DSi)} \right)\|V_\ell-\Pi_{1,\ell} V\|_{L^2(\DSi)}\\& \nonumber\lesssim \left(\|\nabla e_{2,h}\|_{L^2(\DSi)} \|\nabla  u\|_{L^2(\DSi)}\right.\\ \nonumber
 		& \left. +\|\nabla e_{3,h} \|_{L^2(\DSi)} \|\nabla v \|_{L^2(\DSi)} \right)\|V_\ell-\Pi_{1,\ell} V\|_{L^2(\DSi)}
 		\end{align} }
 Moreover, using the approximation and stability properties of $\Pi_{i,\ell} u$, i=1,2,3, we can write
 \begin{align}
 \|(\hat{A} \nabla (V-\Pi_{1,\ell}V), \hat{A} \nabla e_{1,\ell}) \|_{L^2(\DSi)}&\lesssim h  \| \nabla e_{1,\ell}\|_{L^2(\DSi)}\| V\|_{H^{\blue{2}} ({\DSi})},
 \end{align}
 \begin{align}\nonumber
	\left| \left( e^{-\frac{\Pi_{1,\ell} V}{U_T}}  \Pi_{3,\ell} v -e^{-\frac{V}{U_T}}v , e_{1,\ell} \right)_{\DSi}\right|  &\lesssim \| \Pi_{1,\ell} V-V\|_{L^2(\DSi)}\| e_{1,\ell}\|_{L^2(\DSi)}\\
	\nonumber &+\|\Pi_{3,\ell} v-v\|_{L^2(\DSi)}\|  e_{1,\ell}\|_{L^2(\DSi)}\\&\nonumber\lesssim h \|V \|_{H^{1}(\text{D}_{\text{Si}})} \| e_{1,\ell}\|_{L^2(\DSi)}\\&+ h \|v \|_{H^{\blue{1}}(\text{D}_{\text{Si}})} \| e_{1,\ell}\|_{L^2(\DSi)}.
	\end{align}
  \begin{align}\nonumber
	\left| \left( e^{-\frac{\Pi_{1,\ell} V}{U_T}}  \Pi_{3,\ell} v -e^{-\frac{V_\ell}{U_T}}v_\ell , e_{1,\ell} \right)_{\DSi}\right|  &\lesssim \| \Pi_{1,\ell} V-V_\ell\|_{L^2(\DSi)}\| e_{1,\ell}\|_{L^2(\DSi)}\\
	\nonumber &+\|\Pi_{3,\ell} v-v _\ell\|_{L^2(\DSi)}\|  e_{1,\ell}\|_{L^2(\DSi)}\\&\lesssim h \|V \|_{H^{\blue{2}}(\text{D}_{\text{Si}})} \| e_{1,\ell}\|_{L^2(\DSi)},
	\end{align}
 \begin{align}
\nonumber \left| \left( e^{\frac{-\Pi_{1,\ell} V}{U_T}} \nabla \Pi_{3,\ell} v -e^{\frac{-V}{U_T}}\nabla v ,\nabla e_{3,\ell} \right)_{\DSi}\right| &\lesssim \| \Pi_{1,\ell} V-V\|_{L^2(\DSi)}\|\nabla e_{3,\ell}\|_{L^2(\DSi)}\\
\nonumber &+\| \nabla (\Pi_{3,\ell} v-v)\|_{L^2(\DSi)}\| \nabla e_{3,\ell}\|_{L^2(\DSi)}\\&\nonumber\lesssim h \|V \|_{H^{\blue{2}}(\text{D}_{\text{Si}})} \|\nabla e_{3,\ell}\|_{L^2(\DSi)}\\&+ h \|v \|_{H^{\blue{2}}(\text{D}_{\text{Si}})} \|\nabla e_{3,\ell}\|_{L^2(\DSi)}.
 \end{align}
 Also, we can get the similar results for $\left|\left( e^{\frac{\Pi_{1,\ell} V}{U_T}}  \Pi_{2,\ell} u -e^{\frac{V}{U_T}}u , e_{1,\ell} \right)_{\DSi} \right| $, $\left| \left( e^{\frac{\Pi_{1,\ell} V}{U_T}}  \nabla \Pi_{2,\ell} u-e^{\frac{V}{U_T}}\nabla u ,\nabla e_{2,\ell} \right)_{\DSi}\right| $ as well as  $ \left| \left( e^{\frac{\Pi_{1,\ell} V}{U_T}}  \Pi_{2,\ell} u -e^{\frac{V_\ell}{U_T}}u_\ell , e_{1,\ell} \right)_{\DSi}\right|$.
 \\ {\em Step 4:} 
Combination of the above arguments gives us  
\begin{align}\label{e2}\nonumber
\|\hat{A}~\nabla e_{1,\ell}\|_{L^2(\text{D})}^2+\|\nabla e_{2,\ell} \|^2_{L^2(\DSi)}+\|\nabla e_{3,\ell}\| ^2_{L^2(\DSi)} 
&\lesssim h^{\blue{2}} \left( \|V\|^2_{H^{\blue{2}} (\text{D}_{\text{Si}})}+  \|u\|^2_{H^{\blue{2}} (\text{D}_{\text{Si}})}+ \|v\|^2_{H^{\blue{2}} (\text{D}_{\text{Si}})}\right)
%
%
\end{align}
Finally, using the triangle inequality yields the desired results. 
 \end{proof}

 In the next step, we use a residual based a-posteriori error estimation technique to
 estimate the local error $\eta_T$ on each finite element $T\in
 \mathcal{T}_\ell$. The error indicator will serve as the foundation for a
 refinement strategy in order to control and minimize the errors in the
 Poisson equation (\ref{modeleqn}a--\ref{modeleqn}b) and in the
 continuity equations (\ref{modeleqn}c--\ref{modeleqn}d).

 		\label{cor3}
 		{\fontsize{9.2}{9.2} \begin{equation}\nonumber 
 			\end{equation}}
 
 \vspace{-1cm}
 \begin{theorem}[A-posteriori error estimate]
 			\label{theo3}
   For $\omega\in\Omega$ let $(V(\omega),u(\omega),v(\omega))\in \left(H^1(D\setminus
     \Gamma)\cap L^\infty(D\setminus\Gamma)\right)\times\left(
     H^1(\DSi)\cap L^\infty(\DSi)\right)^2$ be the solution of
   \eqref{weak} and $(V_\ell(\omega), u_\ell(\omega), v_\ell(\omega))\in
   X_{1,\ell}\times X_{2,\ell}\times X_{3,\ell}$ be the solution of
   \eqref{weak1}. Then, for sufficiently small $h$, there exist
   constants $c_i$, $i\in\{1,\ldots,6\}$, depending on the doping concentration as well as the shape regularity of the mesh such that
   \begin{align} \label{post}\nonumber
     \left\| {\nabla} \left( {{V_\ell} - V} \right)(\omega) \right\|_{L^2(\text{D})} ^2 + \left\|  \nabla\left( {{u_\ell} - u)}(\omega) \right) \right\|_{L^2(\DSi)} ^2 + \left\| \nabla \left( {{v_\ell} - v} \right)(\omega) \right\|_{L^2(\DSi)} ^2  &\le c_1\sum\limits_{{\zeta\in {\mathcal{T}_\ell}}} {h_{\zeta }^2{{\left\| {{r_1(\omega)}} \right\|^{\blu{2}}}_\zeta}}  + c_2\sum\limits_{{\zeta\in\TSi}} {h_{\zeta}^2{{\left\| {{r_2}(\omega)} \right\|^{\blu{2}}}_\zeta}}    \\\nonumber                                                                                                                                                       &+c_3\sum\limits_{{\zeta\in\TSi}} {h_{{\zeta }}^2{{\left\| {{r_3(\omega)}} \right\|^{\blu{2}}}_{\zeta}}}   + c_4\sum\limits_{\gamma  \in \partial\mathcal{T}_\ell } {\left| {{{\left[ {\hat{A} \frac{{\partial {V_\ell}}}{{\partial \nu }}} (\omega)\right]}_\gamma }} \right|^{\blu{2}}} h_\gamma ^2 \\
                                                                                                                                                            &+c_5 \sum\limits_{\gamma  \in \partial \TSi } {\left| {{{\left[ {{U_T}{\mu _n}\frac{{\partial {u_\ell}}}{{\partial \nu }}}(\omega) \right]}_\gamma }} \right|^{\blu{2}}} h_\gamma ^2+c_6\sum\limits_{\gamma  \in \partial \TSi } {\left| {{{\left[ {{U_T}{\mu _p}\frac{{\partial {v_\ell}}}{{\partial \nu }}}(\omega) \right]}_\gamma }} \right|^{\blu{2}} } h_\gamma ^2
   \end{align}
   holds for a fixed event $\omega\in \Omega$ where $\TSi:=\mathcal{T}_\ell\cap\DSi$ and
   \begin{align*}
     r_1(\omega)&:=- \nabla \cdot \left( {A\nabla {V_\ell(\omega)}} \right) - q{C(\omega)} + q{n_i}\left( {{e^{\frac{{{V_\ell(\omega)}}}{{{U_T}}}}}u_{\blue{\ell}}(\omega) - {e^{ - \frac{{{V_\ell(\omega)}}}{{{U_T}}}}}{v_\ell(\omega)}} \right),\\
     {r_2}(\omega)&: = {U_T}\nabla \cdot \left( {{\mu _n}{e^{\frac{{{V_\ell(\omega)}}}{{{U_T}}}}}\nabla {u_\ell(\omega)}} \right) - \frac{{{u_\ell(\omega)}{v_\ell(\omega)} - 1}}{{{\tau _p}\left( {{e^{\frac{{{V_\ell(\omega)}}}{{{U_T}}}}}{u_\ell(\omega)} + 1} \right) + {\tau _n}\left( {{e^{ - \frac{{{V_\ell(\omega)}}}{{{U_T}}}}}{v_\ell(\omega)} + 1} \right)}},\\
     {r_3}(\omega)&:= {U_T}\nabla \cdot  \left( {{\mu _p}{e^{-\frac{{{V_\ell(\omega)}}}{{{U_T}}}}}\nabla {v_\ell(\omega)}} \right) - \frac{{{u_\ell(\omega)}{v_\ell(\omega)} - 1}}{{{\tau _p}\left( {{e^{\frac{{{V_\ell(\omega)}}}{{{U_T}}}}}{u_\ell(\omega)} + 1} \right) + {\tau _n}\left( {{e^{ - \frac{{{V_\ell}(\omega)}}{{{U_T}}}}}{v_\ell(\omega)} + 1} \right)}}.
   \end{align*}
   Here $\zeta$ denotes the area of an element in $\mathcal{T}_\ell$ or
   $\TSi$, $\gamma$ denotes the boundary of the element, and the brackets
   $\left[\cdot\right]$ indicate the jump at the element boundary.
 \end{theorem}
 
 \begin{proof}
   In the following, the dependence of the solutions on the random
   variable~$\omega$ is not indicated in order to simplify notation.
   We first define
   \begin{equation*}
     {\varepsilon_{1,\ell}}: = V - {V_\ell},
     \qquad{\varepsilon_{2,\ell}}: = u - {u_\ell},
     \qquad{\varepsilon_{3,\ell}}: = v - {v_\ell}.
   \end{equation*}
   Using the test functions $\varphi_1 \in X_1$, $\varphi_2 \in X_2$, and
   $\varphi_3 \in X_3$, the weak formulation \eqref{weak} yields
   \begin{subequations}
     {\fontsize{9.2}{9.2}
     \begin{align}
       \label{weak11}\nonumber
       {\left( {\hat{A} \nabla  {\varepsilon_{1,\ell}},\hat{A} \nabla  {\varphi _1}} \right)_{\DSi}} =&~ {\left( {q{C},{\varphi _1}} \right)_{\DSi}} - {\left( {q{n_i}{e^{\frac{V}{{{U_T}}}}}u - q{n_i}{e^{\frac{{{V_\ell}}}{{{U_T}}}}}{u_\ell},{\varphi _1}} \right)_{\DSi}}  \\
       +&~{\left( {q{n_i}{e^{ - \frac{V}{{{U_T}}}}}v - q{n_i}{e^{ - \frac{{{V_\ell}}}{{{U_T}}}}}{v_\ell},{\varphi _1}} \right)_{\DSi}} - {\left( {q{n_i}{e^{\frac{{{V_\ell}}}{{{U_T}}}}}{u_\ell} - q{n_i}{e^{ - \frac{{{V_\ell}}}{{{U_T}}}}}{v_\ell},{\varphi _1}} \right)_{\DSi}}- {\left( {\hat{A} \nabla  {V_{h}},\hat{A} \nabla  {\varphi _1}} \right)_{\DSi}},\\   \label{weak12}
       {\left( {\hat{A} \nabla  {\varepsilon_{1,\ell}},\hat{A} \nabla  {\varphi _1}} \right)_{\Dox}} =&~  - {\left( {\hat{A} \nabla  {V_\ell},\hat{A} \nabla  {\varphi _1}} \right)_{\Dox}}, \\\nonumber
        {\left( {{U_T}\mu_n{e^{\frac{V}{{{U_T}}}}}\nabla   u - {U_T}{\mu_n}{e^{\frac{{{V_\ell}}}{{{U_T}}}}}\nabla  {u_\ell},\nabla  {\varphi _2}} \right)_{{\DSi}}} =-& ~{\left( {{U_T}\mu _n{e^{\frac{{{V_\ell}}}{{{U_T}}}}}\nabla  {u_\ell},\nabla  {\varphi _2}} \right)_{{\DSi}}} \\\nonumber
       -&~{\left( {\frac{{uv - 1}}{{{\tau _p}\left( {{e^{\frac{V}{{{U_T}}}}}u + 1} \right) + \tau_n\left( {{e^{ - \frac{V}{{{U_T}}}}}v + 1} \right)}} - \frac{{{u_\ell}{v_\ell} - 1}}{{{\tau_p}\left( {{e^{\frac{{{V_\ell}}}{{{U_T}}}}}{u_\ell} + 1} \right) + \tau_n\left( {{e^{ - \frac{{{V_\ell}}}{{{U_T}}}}}{v_\ell} + 1} \right)}},{\varphi _2}} \right)_{{\DSi}}} \\  \label{weak13}
       -&~{\left( {\frac{{{u_\ell}{v_\ell} - 1}}{{\tau_p\left( {{e^{\frac{{{V_\ell}}}{{{U_T}}}}}{u_\ell} + 1} \right) + {\tau _n}\left( {{e^{ - \frac{{{V_\ell}}}{{{U_T}}}}}{v_\ell} + 1} \right)}},{\varphi _2}} \right)_{{\DSi}}},\\\nonumber
        {\left( {{U_T}{\mu _p}{e^{\frac{-V}{{{U_T}}}}}\nabla   v - {U_T}{\mu _p}{e^{\frac{{{-V_\ell}}}{{{U_T}}}}}\nabla  {v_\ell},\nabla  {\varphi _3}} \right)_{{\DSi}}} =-&~ {\left( {{U_T}{\mu _p}{e^{\frac{{{-V_\ell}}}{{{U_T}}}}}\nabla  {v_\ell},\nabla  {\varphi _3}} \right)_{{\DSi}}} \\\nonumber
       -&~	{\left( {\frac{{uv - 1}}{{{\tau_p}\left( {{e^{\frac{V}{{{U_T}}}}}u + 1} \right) + \tau_n\left( {{e^{ - \frac{V}{{{U_T}}}}}v + 1} \right)}} - \frac{{{u_\ell}{v_\ell} - 1}}{{{\tau _p}\left( {{e^{\frac{{{V_\ell}}}{{{U_T}}}}}{u_\ell} + 1} \right) + {\tau _n}\left( {{e^{ - \frac{{{V_\ell}}}{{{U_T}}}}}{v_\ell} + 1} \right)}},{\varphi _3}} \right)_{{\DSi}}}\\ \label{weak14}
       -&~	{\left( {\frac{{{u_\ell}{v_\ell} - 1}}{{{\tau _p}\left( {{e^{\frac{{{V_\ell}}}{{{U_T}}}}}{u_\ell} + 1} \right) + {\tau_n}\left( {{e^{ - \frac{{{V_\ell}}}{{{U_T}}}}}{v_\ell} + 1} \right)}},{\varphi _3}} \right)_{{\DSi}}}.
     \end{align}}
   \end{subequations}
   Let $I_{i,\ell}$ be the projection operator on $X_{i,\ell}$ defined in Theorem 4.8.7 of \cite{brenner2007mathematical}.
   
   Next we substitute $\psi_1:=I_{1,\ell}{\varphi_1} $ into (\ref{weak1}a) and
   (\ref{weak1}b), $\psi_2:=I_{2,\ell}{\varphi _2} $ into (\ref{weak1}c), and
   $\psi_3:=I_{3,\ell}{\varphi _3}$ into (\ref{weak1}d), and sum up the
   equations. Then we subtract them 
    from (\ref{weak11})--(\ref{weak14}), which leads to
 {\fontsize{9.4}{9.4}
   \begin{equation}\label{aa}
     \begin{array}{l}
       {\left( {\hat{A} \nabla  {\varepsilon_{1,\ell}},\hat{A} \nabla  {\varphi _1}} \right)_{L^2(\text{D})} } + {\left( {{U_T}{\mu_n}{e^{\frac{V}{{{U_T}}}}}\nabla   u - {U_T}{\mu_n}{e^{\frac{{{V_\ell}}}{{{U_T}}}}}\nabla  {u_\ell},\nabla  {\varphi _2}} \right)_{{\DSi}}} - {\left( {{U_T}{\mu_p}{e^{-\frac{V}{{{U_T}}}}}\nabla   v - {U_T}{\mu _p}{e^{-\frac{{{V_\ell}}}{{{U_T}}}}}\nabla  {v_\ell},\nabla  {\varphi _3}} \right)_{{\DSi}}}\\[3 mm]
       + {\left( {q{n_i}{e^{\frac{V}{{{U_T}}}}}u - q{n_i}{e^{\frac{{{V_\ell}}}{{{U_T}}}}}{u_\ell},{\varphi _1}} \right)_{\DSi}} + {\left( {q{n_i}{e^{ - \frac{V}{{{U_T}}}}}v - q{n_i}{e^{ - \frac{{{V_\ell}}}{{{U_T}}}}}{v_\ell},{\varphi _1}} \right)_{\DSi}} 
       = {\left( {q{C},{\varphi _1} - {I_{1,\ell}}{\varphi _1}} \right)_{\DSi}} - {\left( {q{n_i}{e^{\frac{{{V_\ell}}}{{{U_T}}}}}{u_\ell} - q{n_i}{e^{ - \frac{{{V_\ell}}}{{{U_T}}}}}{v_\ell},{\varphi _1} - {I_{1,\ell}}{\varphi _1}} \right)_{\DSi}} \\[3 mm]
       -{\left( {\hat{A} \nabla  {V_\ell},\hat{A} \nabla  \left( {{\varphi _1} - {I_{1,\ell}}{\varphi _1}} \right)} \right)_{\Dox} }  + 
       {\left( {{U_T}\mu_n{e^{\frac{{{V_\ell}}}{{{U_T}}}}}\nabla  {u_\ell},\nabla  \left( {I_{2,\ell}{\varphi _2} - {\varphi _2}} \right)} \right)_{{\DSi}}} + {\left( {{U_T}{\mu _p}{e^{-\frac{{{V_\ell}}}{{{U_T}}}}}\nabla  {v_\ell},\nabla  \left( {I_{3,\ell}{\varphi _3} -{\varphi_3}} \right)} \right)_{{\DSi}}}  \\[3 mm]
       -{\left( {\frac{{uv - 1}}{{\tau_p\left( {{e^{\frac{V}{{{U_T}}}}}u + 1} \right) + \tau_n\left( {{e^{ - \frac{V}{{{U_T}}}}}v + 1} \right)}} - \frac{{{u_\ell}{v_\ell} - 1}}{{{\tau _p}\left( {{e^{\frac{{{V_\ell}}}{{{U_T}}}}}{u_\ell} + 1} \right) + \tau_n\left( {{e^{ - \frac{{{V_\ell}}}{{{U_T}}}}}{v_\ell} + 1} \right)}},{\varphi _2}} \right)_{{\DSi}}} - 
       {\left( {\frac{{uv - 1}}{{\tau_p\left( {{e^{\frac{V}{{{U_T}}}}}u + 1} \right) + \tau_n\left( {{e^{ - \frac{V}{{{U_T}}}}}v + 1} \right)}} - \frac{{{u_\ell}{v_\ell} - 1}}{{\tau _p\left( {{e^{\frac{{{V_\ell}}}{{{U_T}}}}}{u_\ell} + 1} \right) + \tau_n\left( {{e^{ - \frac{{{V_\ell}}}{{{U_T}}}}}{v_\ell} + 1} \right)}},{\varphi _3}} \right)_{{\DSi}}}  \\[3 mm]
       +{\left( {\frac{{{u_\ell}{v_\ell} - 1}}{{{\tau _p}\left( {{e^{\frac{{{V_\ell}}}{{{U_T}}}}}{u_\ell} + 1} \right) + {\tau_n}\left( {{e^{ - \frac{{{V_\ell}}}{{{U_T}}}}}{v_\ell} + 1} \right)}},{I_{2,\ell}\varphi _2} - {\varphi _2}} \right)_{{\DSi}}} + {\left( {\frac{{{u_\ell}{v_\ell} - 1}}{{{\tau _p}\left( {{e^{\frac{{{V_\ell}}}{{{U_T}}}}}{u_\ell} + 1} \right) + \tau_n\left( {{e^{ - \frac{{{V_\ell}}}{{{U_T}}}}}{v_\ell} + 1} \right)}},{I_{2,\ell}\varphi _3} - {\varphi _3}} \right)_{{\DSi}}}.
     \end{array}
   \end{equation} }
 
   Explicit error estimation involves the direct computation of the
   interior element residuals and the jumps at the element boundaries
   to find an estimate for the error. Now using Green's theorem and the
   Cauchy-Schwarz inequality on \eqref{aa} leads to 
 {\fontsize{9.25}{9.25}
   \begin{align}\label{31}\nonumber
      & {\left( {\hat{A} \nabla {\varepsilon_{1,\ell}},\hat{A} \nabla {\varphi _1}} \right)_{L^2(\text{D})} }- {\left( {{U_T}{\mu _n}{e^{\frac{V}{{{U_T}}}}}\nabla u - {U_T}{\mu _n}{e^{\frac{{{V_\ell}}}{{{U_T}}}}}\nabla {u_\ell},\nabla {\varphi _2}} \right)_{{\DSi}}} - {\left( {{U_T}{\mu _n}{e^{-\frac{V}{{{U_T}}}}}\nabla v - {U_T}{\mu _n}{e^{-\frac{{V_\ell}}{{{U_T}}}}}\nabla {v_\ell},\nabla {\varphi _3}} \right)_{{\DSi}}}\\[1 mm]\nonumber
      & + {\left( {q{n_i}{e^{\frac{V}{{{U_T}}}}}u - q{n_i}{e^{\frac{{{V_\ell}}}{{{U_T}}}}}{u_\ell},{\varphi _1}} \right)_{\DSi}} + {\left( {q{n_i}{e^{ - \frac{V}{{{U_T}}}}}v - q{n_i}{e^{ - \frac{{{V_\ell}}}{{{U_T}}}}}{v_\ell},{\varphi _1}} \right)_{\DSi}} \le \sum\limits_{{\zeta\in\mathcal{T}_\ell}} {{{\left\| {{r_1}} \right\|}_{{\zeta}}}} {\left\| {{\varphi _1} - {I_{1,\ell}}{\varphi _1}} \right\|_{{\zeta}}} + \sum\limits_{{\zeta\in\TSi}} {{{\left\| {{r_2}} \right\|}_{{\zeta _i}}}} {\left\| {{\varphi _2} - {I_{2,\ell}}{\varphi _2}} \right\|_{{\zeta}}}  \\[1  mm]
      & +\sum\limits_{{\zeta\in\TSi}} {{{\left\| {{r_3}} \right\|}_{{\zeta }}}} {\left\| {{\varphi _3} - {I_{2,\ell}}{\varphi _3}} \right\|_{{\zeta}}} + \sum\limits_{\gamma  \in \partial\mathcal{T}_\ell } {\left| {{{\left[ {\hat{A} \frac{{\partial {V_\ell}}}{{\partial \nu }}} \right]}_\gamma }} \right|} {\left\| {{\varphi _1} - {I_{1,\ell}}{\varphi _1}} \right\|_{{\gamma}}} + \sum\limits_{\gamma  \in \partial\TSi } {\left| {{{\left[ {{U_T}{\mu _n}\frac{{\partial {u_\ell}}}{{\partial \nu }}} \right]}_\gamma }} \right|} {\left\| {{\varphi _2} - {I_{2,\ell}}{\varphi _2}} \right\|_\gamma } 
       +\sum\limits_{\gamma  \in \partial\TSi } {\left| {{{\left[ {{U_T}{\mu _p}\frac{{\partial {v_\ell}}}{{\partial \nu }}} \right]}_\gamma }} \right|} {\left\| {{\varphi _3} - {I_{3,\ell}}{\varphi _3}} \right\|_\gamma },
   \end{align}}
   where
   \begin{align*}
     r_1: &=  - \nabla \cdot \left( {A\nabla {V_\ell}} \right) - q{C} + q{n_i}\left( {{e^{\frac{{{V_\ell}}}{{{U_T}}}}}u_{\ell} - {e^{ - \frac{{{V_\ell}}}{{{U_T}}}}}{v_\ell}} \right),\\
     {r_2}:& = {U_T}\nabla \cdot  \left( {{\mu _n}{e^{\frac{{{V_\ell}}}{{{U_T}}}}}\nabla {u_\ell}} \right) - \frac{{{u_\ell}{v_\ell} - 1}}{{{\tau _p}\left( {{e^{\frac{{{V_\ell}}}{{{U_T}}}}}{u_\ell} + 1} \right) + {\tau _n}\left( {{e^{ - \frac{{{V_\ell}}}{{{U_T}}}}}{v_\ell} + 1} \right)}},\\
     {r_3}: &= {U_T}\nabla  \cdot \left( {{\mu _n}{e^{-\frac{{{V_\ell}}}{{{U_T}}}}}\nabla {v_\ell}} \right) - \frac{{{u_\ell}{v_\ell} - 1}}{{{\tau _p}\left( {{e^{\frac{{{V_\ell}}}{{{U_T}}}}}{u_\ell} + 1} \right) + {\tau _n}\left( {{e^{ - \frac{{{V_\ell}}}{{{U_T}}}}}{v_\ell} + 1} \right)}}.
   \end{align*}
   It can be easily seen that  
   \begin{subequations}
     \begin{align}\label{a1}
       {e^{\frac{V}{{{U_T}}}}} u - {e^{\frac{{{V_\ell}}}{{{U_T}}}}}{u_\ell} &= {e^{\frac{V}{{{U_T}}}}}{\varepsilon_{2,\ell}} - \left( {{e^{\frac{V}{{{U_T}}}}} - {e^{\frac{{{V_\ell}}}{{{U_T}}}}}} \right){\varepsilon_{2,\ell}} + \left( {{e^{\frac{V}{{{U_T}}}}} - {e^{\frac{{{V_\ell}}}{{{U_T}}}}}} \right)u,\\\label{a2}
       {e^{\frac{V}{{{U_T}}}}}\nabla   u - {e^{\frac{{{V_\ell}}}{{{U_T}}}}}\nabla  {u_\ell} &= {e^{\frac{V}{{{U_T}}}}}\nabla  {\varepsilon_{2,\ell}} - \left( {{e^{\frac{V}{{{U_T}}}}} - {e^{\frac{{{V_\ell}}}{{{U_T}}}}}} \right)\nabla  {\varepsilon_{2,\ell}} + \left( {{e^{\frac{V}{{{U_T}}}}} - {e^{\frac{{{V_\ell}}}{{{U_T}}}}}} \right)u,\\\label{a3}
       {e^{\frac{V}{{{U_T}}}}} v - {e^{-\frac{{{V_\ell}}}{{{U_T}}}}}{v_\ell} &= {e^{-\frac{V}{{{U_T}}}}}{\varepsilon_{3,\ell}} - \left( {{e^{-\frac{V}{{{U_T}}}}} - {e^{-\frac{{{V_\ell}}}{{{U_T}}}}}} \right){\varepsilon_{3,\ell}} + \left( {{e^{-\frac{V}{{{U_T}}}}} - {e^{-\frac{{{V_\ell}}}{{{U_T}}}}}} \right)v,\\\label{a4}
       {e^{-\frac{V}{{{U_T}}}}}\nabla   v - {e^{-\frac{{{V_\ell}}}{{{U_T}}}}}\nabla  {v_\ell} &= {e^{-\frac{V}{{{U_T}}}}}\nabla  {\varepsilon_{3,\ell}} - \left( {{e^{-\frac{V}{{{U_T}}}}} - {e^{-\frac{{{V_\ell}}}{{{U_T}}}}}} \right)\nabla  {\varepsilon_{3,\ell}} + \left( {{e^{-\frac{V}{{{U_T}}}}} - {e^{-\frac{{{V_\ell}}}{{{U_T}}}}}} \right)v.
     \end{align}
   \end{subequations}
   Now substituting $\varphi_i={\varepsilon_{i,\ell}}$ for $i \in
   \{1,2,3\}$ into \eqref{31} and using (\ref{a1}--\ref{a4}), Eqs. \eqref{M1},  
   \eqref{inequ11}, \cite[Corollary 4.8.15 on page 123]{brenner2007mathematical}, and Cauchy-Schwarz inequality  yield
   \begin{align}\nonumber
     {\left\|\nabla  \varepsilon_{1,\ell} \right\|}_{L^2(\text{D})}^2+{\left\|\nabla  \varepsilon_{2,\ell} \right\|}_{L^2(\DSi)}^2+{\left\|\nabla  \varepsilon_{3,\ell} \right\|}_{L^2(\DSi)}^2
      \lesssim&~\sum\limits_{{\zeta\in\mathcal{T}_\ell}} {h_{{\zeta}}^2{{\left\| {{r_1}} \right\|^{\blu{2}}}_{{\zeta}}}}  +\sum\limits_{{\zeta\in\TSi}} {h_{{\zeta}}^2{{\left\| {{r_2}} \right\|^{\blu{2}}}_{{\zeta}}}}  +\sum\limits_{{\zeta\in\TSi}} {h_{{\zeta}}^2{{\left\| {{r_3}} \right\|^{\blu{2}}}_{{\zeta}}}}  + \sum\limits_{\gamma  \in \partial\mathcal{T}_\ell } {\left| {{{\left[ {\hat{A} \frac{{\partial {V_\ell}}}{{\partial \nu }}} \right]}_\gamma }} \right|^{\blu{2}}} h_\gamma ^2  \\\nonumber
     +&	\sum\limits_{\gamma  \in \partial\TSi } {\left| {{{\left[ {{U_T}{\mu _n}\frac{{\partial {u_\ell}}}{{\partial \nu }}} \right]}_\gamma }} \right|^{\blu{2}}} h_\gamma ^2 + \sum\limits_{\gamma  \in \partial\TSi } {\left| {{{\left[ {{U_T}{\mu _p}\frac{{\partial {v_\ell}}}{{\partial \nu }}} \right]}_\gamma }} \right|^{\blu{2}}} h_\gamma ^2\\
    +& {{{\left\| \varepsilon_{1,\ell} \right\|}_{L^2(\DSi)}^2 }}+ {{{\left\| \varepsilon_{2,\ell} \right\|}_{L^2(\DSi)}^2 }}+{{{\left\| \varepsilon_{3,\ell} \right\|}_{L^2(\DSi)}^2 }},
   \end{align}
 where $\TSi:=\mathcal{T}_\ell\cap\DSi$.  
%
   Finally, applying an inverse estimate, for sufficiently small~$h$ we get
   \begin{align}\label{eq:indi} \nonumber
       {\left\|{\nabla} \varepsilon_{1,\ell} \right\|}_{L^2(\text{D})}^2+{\left\|  {\nabla}\varepsilon_{2,\ell} \right\|}_{L^2(\DSi)}^2+{\left\| {\nabla} \varepsilon_{3,\ell} \right\|}_{L^2(\DSi)}^2   \le&~ c_1\sum\limits_{{\zeta\in\mathcal{T}_\ell}} {h_{{\zeta}}^2{{\left\| {{r_1}} \right\|^{\blu{2}}}_{{\zeta}}}}  + c_2\sum\limits_{{\zeta\in\TSi}} {h_{{\zeta}}^2{{\left\| {{r_2}} \right\|^{\blu{2}}}_{{\zeta}}}}  + c_3\sum\limits_{{\zeta\in\TSi}} {h_{{\zeta}}^2{{\left\| {{r_3}} \right\|^{\blu{2}}}_{{\zeta}}}}  + c_4\sum\limits_{\gamma  \in \partial\mathcal{T}_\ell } {\left| {{{\left[ {\hat{A} \frac{{\partial {V_\ell}}}{{\partial \nu }}} \right]}_\gamma }} \right|^{\blu{2}}} h_\gamma ^2  \\
     &~+	c_5\sum\limits_{\gamma  \in \partial\TSi } {\left| {{{\left[ {{U_T}{\mu _n}\frac{{\partial {u_\ell}}}{{\partial \nu }}} \right]}_\gamma }} \right|^{\blu{2}}} h_\gamma ^2 + c_6\sum\limits_{\gamma  \in \partial\TSi } {\left| {{{\left[ {{U_T}{\mu _p}\frac{{\partial {v_\ell}}}{{\partial \nu }}} \right]}_\gamma }} \right|^{\blu{2}}} h_\gamma ^2,
   \end{align}
   which concludes the proof.
 \end{proof}

 \section{Multilevel Monte Carlo Finite-Element Method}
  
 In a Monte Carlo finite-element (FE) method several evaluations
 are combined   to obtain an approximation of the solution of
 the model equation or equations.
 \red{Considering the \blue{Bochner} space $L^2(\Omega;X)$ for the mappings $\mathcal{Y}\colon \Omega \to X$, we can define
 \begin{align}\label{bochner norm}
 	\| \mathcal{Y} \|_{L^2(\Omega;X)} :=\Big( \int_\Omega \| \mathcal{Y}(\cdot,\omega) \|_X^2 \rmd\P(\omega)
 	 	\Big)^{1/2}
 	 	= \E\Big[\| \mathcal{Y}(\cdot,\omega) \|_X^2 \Big]^{1/2}.
 	\end{align}}
  The standard MC estimator $\EMC$ for
 $\mathbb{E}[u_\ell]$ is the sample mean
 \begin{equation}\label{MC-estimator}
   \EMC[u_\ell]
   := \hat{u}_\ell
   := \frac{1}{M} \sum_{i=1}^{M} u_\ell^{(i)},
 \end{equation}
 where $u_\ell^{(i)} = u_\ell(x,\omega^{(i)})$ is the $i$-th sample (independent random variable) of the
 solution~$u$.
 	
 
In a multilevel Monte Carlo (MLMC) method, instead of calculating the expected value $\mathbb{E}[u]$ by
 $\mathbb{E}[u_\ell]$ on a constant triangulation $\mathcal{T}_\ell$, the
 MLMC method approximates the expected value $\mathbb{E}[u]$ using
 several $\mathbb{E}[u_\ell]$, $\ell \in \{0,1,\ldots,L\}$,
 estimated on the nested family $\{\mathcal{T}_{{\ell}}
 \}_{{\ell}=0}^{\infty}$. In fact, to overcome the drawback of the MC
 method, the MLMC estimator avoids prohibitively many expensive
 evaluations of $\mathbb{E}[u_\ell]$ on the finest level~$L$.
 
 The FE approximation of the expected value of $u_{L}$ at level~$L$
 can be written as
 \begin{equation}\label{MLMC}
   {\mathbb{E}}[u_{L}]
   ={\mathbb{E}}[u_{0}]+{\mathbb{E}}\left[\sum_{{\ell}=1}^{L}(u_{{\ell}}-u_{{{\ell}-1}})\right]
   ={\mathbb{E}}[u_{0}]+\sum_{{\ell}=1}^{L}{\mathbb{E}}[u_{{\ell}}-u_{{{\ell}-1}}].
 \end{equation}
 Therefore, for different numbers $M_\ell$ of samples at level $\ell
 \in \{0,\ldots,L\}$, we have
 \begin{equation*}
   \EMLMC[u]
   := \hat{u}_{L}
   =\frac{1}{M_0}\sum_{i=1}^{M_0}u_{0}^{(i)}
   +\sum_{{\ell}=1}^{L}\frac{1}{M_{\ell}}\sum_{i=1}^{M_{\ell}}(u_{{\ell}}^{(i)}-u_{{{\ell}-1}}^{(i)}).
 \end{equation*}
 The mean square error (MSE) is estimated by
 \begin{equation}\label{MLMC error1}
   \begin{split}
     \MSE 
     &\leq M_0^{-1} \sigma^2[ u_{0} ]
     + \sum_{\ell=1}^{L} M_\ell^{-1} \sigma^2[ u_{\ell} - u_{{\ell-1}} ]
     +\| \E[u_{L}]  -\E[u] \|^2_{L^2(\Omega;D)},
   \end{split}
 \end{equation}
 as shown for example in \cite{taghizadeh2017optimal}, where the
 variance is given by $\sigma[u]^2 := \|\E[u]-u\|_{L^2(\Omega;D)}^2$. The
 first and second terms of \eqref{MLMC error1} are the statistical
 error, while the last term is the discretization error.
 In the next section, we study the effect of mesh refinement, i.e.,
 uniform refinement  and adaptive
 refinement (using the error indicator \eqref{eq:indi}) on both terms
 of errors.

Here, we strive to use the a priori and a posteriori error
  estimates obtained above   for the stochastic
  drift-diffusion-Poisson system. To do so,
	we consider  the convexity of the norms and employ Jensen's
        inequality as well as the Poincar\'e inequality. Therefore,
        taking the expectation in Theorems~\ref{theo2} and~\ref{theo3}
        with respect to a random variable leads to the following corollaries.

 \begin{corollary}
 		\label{cor1}
 	The error bound
 		{\fontsize{9.4}{9.4} \begin{equation}\label{error11}
 			\left\| {  {\E[ e_{1,\ell}]}} \right\|_{L^2(D)}^2  + \left\| { {\E[ e_{2,\ell}]}} \right\|_{L^2(\DSi)}^2+
 			\left\| {\E[{e_{3,\ell}]}} \right\|_{L^2(\DSi)}^2\leq  \E\left(\left\| {  {e_{1,\ell}}} \right\|_{L^2(\text{D})} ^2 + \left\| {{e_{2,\ell}}} \right\|_{L^2(\DSi)} ^2 + \left\| {{e_{3,\ell}}} \right\|_{L^2(\DSi)} ^2 \right) \lesssim {h^{2}}\red{\mathbb{E}}\left( {\left\| V \right\|_{H^2 (\text{D}_{\text{Si}})} ^2 +\left\| u \right\|_{H^2 (\text{D}_{\text{Si}})} ^2 + \left\| v \right\|_{H^2 (\text{D}_{\text{Si}})} ^2   } \right),
                      \end{equation}}
                    holds for  a priori error estimation.
 	\end{corollary}
 \begin{corollary}
 		\label{coro2}
 	 	The error bound
 		{\fontsize{9.7}{9.7}	\begin{align} 		\label{coro21}\nonumber
 			\left\| {  {\E[\varepsilon_{1,\ell}]}} \right\|_{L^2(D)}^2  + \left\| { {\E[\varepsilon_{2,\ell}]}} \right\|_{L^2(\DSi)}^2+
 			\left\| {\E[{\varepsilon_{3,\ell}]}} \right\|_{L^2(\DSi)}^2&\leq  \E\left(\left\| {   {\varepsilon_{1,\ell}}} \right\|_{L^2(\text{D})} ^2 + \left\| {{\varepsilon_{2,\ell}}} \right\|_{L^2(\DSi)} ^2 + \left\| {{\varepsilon_{3,\ell}}} \right\|_{L^2(\DSi)} ^2 \right) \blue{\le} \\\nonumber
 			&~~~c_1\sum\limits_{{\zeta\in\mathcal{T}_\ell}}  h_\zeta^2 \E\left[ \| r_1(\omega)\, \|_\zeta^2 \right]+ c_2\sum\limits_{{\zeta\in\TSi}} {h_{{\zeta}}^2 \E\left[ \| r_2 (\omega)\,\|_\zeta^2 \right]}  + c_3\sum\limits_{{\zeta\in\TSi}} {h_{{\zeta}}^2\E\left[ \| r_3 (\omega)\,\|_\zeta^2 \right]}\\ \nonumber &+ c_4\sum\limits_{\gamma  \in \partial\mathcal{T}_\ell } {\E\,\left| {{{\left[ {\hat{A} \frac{{\partial {V_\ell}}}{{\partial \nu }}}(\omega) \right]}_\gamma }}\, \right|^2} h_{\gamma} ^2  
 			~+	c_5\sum\limits_{\gamma  \in \partial\TSi } \E\,{\left| {{{\left[ {{U_T}{\mu _n}\frac{{\partial {u_\ell}}}{{\partial \nu  }}} (\omega)\right]}_\gamma }}\, \right|^2} h_{\gamma} ^2 \\
 			&+ c_6\sum\limits_{\gamma  \in \partial\TSi } {\E\left| {{{\left[ {{U_T}{\mu _p}\frac{{\partial {v_\ell}}}{{\partial \nu }}} \right]}_\gamma }} \right|^2} h_{\gamma} ^2,
 			\end{align}}
                      holds for a posteriori error estimation.
 	\end{corollary}
 	 Corollary \ref{coro2} makes it possible to estimate the a \red{posteriori} error estimator for
 	the expected value of the solutions. Marking strategies such as the
 	D\"orfler strategy \cite{dorfler1996convergent} can now be used to
 	drive mesh adaptivity.
       
 In order to estimate the computational errors, we continue with the
 degrees of freedom. It enables us to draw a fair comparison between
 adaptive MLMC-FE and uniform MLMC-FE methods. According to the error
 bound given in Corollary \ref{cor1}, we define the discretization error as  
 \begin{equation}\label{disc}
   \mathcal{E}_\ell:=  \left(\left\| {  {\E[e_{1,\ell}]}} \right\|_{L^2(D)}^2  + \left\| { {\E[e_{2,\ell}]}} \right\|_{L^2(\DSi)}^2+
     \left\| {\E[{e_{3,\ell}]}} \right\|_{L^2(\DSi)}^2\right)^{1/2}.
 \end{equation}
Furthermore, at level~$\ell$,  we assume that
 \begin{equation}
   \mathcal{E}_\ell^2 \leq C_1 \mathcal{N}_\ell ^{-2\alpha}
   :=C_1\left(\mathcal{N}_{\mathcal{P}_\ell}+2\mathcal{N}_{\mathcal{D}_\ell}\right)^{-2\alpha}
   \label{eq:l2error}
 \end{equation}
where $\mathcal{N}_{\mathcal{P}_\ell}$ is the number of
 unknowns or the (degrees of freedom) for the Poisson equation and
 $\mathcal{N}_{\mathcal{D}_\ell}$ indicates the number of unknowns for
 the two continuity equations. The exponent~$\alpha$ is the
 convergence rate of the error. For the statistical error, the following
 inequality
 \begin{equation}
   \sigma^2[\Delta V_{h_\ell}]
   +  \sigma^2[\Delta u_{h_\ell}]
   +  \sigma^2[\Delta v_{h_\ell}]
   \le C_2\mathcal{N}_\ell^{-\beta}
   \label{eq:variance}
 \end{equation}
 is assumed to show the convergence of the statistical error (at level~$\ell$). For $\ell=0$, the assumption
 \begin{equation*}
   \sigma^2[\Delta V_{h_0}]
   +  \sigma^2[\Delta u_{h_0}]
   +  \sigma^2[\Delta v_{h_0}]
   \le C_0
 \end{equation*}
 is used as well.
 
 Due to the computational challenge of solving a system of
 SPDEs, an effective computational
 strategy is crucial. We strive to determine the optimal number
 $M_{\ell}$ of samples which minimize the computational work when $\MSE \le \varepsilon^2$. In other words, the optimal number of samples are defined such that the statistical error is less than $\varepsilon^2/2$. The optimal value of $L$ (the lowest possible number) determined in the sense that the discretization error ($\mathcal{E}_L$) is less than $\varepsilon/\hspace{-0.05cm}\sqrt{2}$. 
  For this,  the following
 optimization problem is solved
 \begin{equation}
   \begin{aligned}
     & \underset{M_\ell}{\text{minimize}}
     & & f(M_\ell) :=
     \sum_{\ell=0}^{L}  M_\ell\mathcal{N}_\ell,\\
     & \text{subject to}
     & & g(M_\ell):= \frac{ C_{0}} { M_0} + C_2 \sum_{\ell=1}^{L} \frac{\mathcal{N}_\ell^{-\beta}  }{M_\ell }   \leq  \frac{\varepsilon^2}{2},
     \label{op:MLMC}
   \end{aligned}
 \end{equation}
 where the optimization is over all $M_\ell>0$. Moreover, since the
 optimal numbers $M_\ell$ of samples at level~$\ell$ are in general not
 integers, they are rounded up and replaced by $\lceil M_\ell \rceil$. The details of the optimal approach were given in Giles's MLMC paper \cite{giles2009multilevel}.
 
 \section{Numerical Example and Results}
  	 	 	
 In this work, we have chosen a double-gate MOSFET (DG-MOSFET) as a
 realistic example to implement the multilevel adaptive method
 developed above and to investigate its behavior. In these
 semiconductor transistors, the width of the silicon channel is very
 small and two gate contacts are used in the both sides of the channel
 to control the channel efficiently \cite{taur2004continuous,
   wong1997self, munteanu2006quantum}. Hence, the current can
 potentially be twice the current through a single-gate device, since
 inversion layers can exist at both gates. This device structure
 suppresses short-channel effects and leads to higher currents as
 compared to the usual MOSFET structure having only one gate. In
   this work we assumed that there is no charge fluctuation in the
   oxide layer (see \ref{modeleqn}b). However, due to charges in the
   dielectric subdomain and its boundaries (interface to the channel
   region), the right-hand side in Eq.~\ref{modeleqn}b can also be
   assumed to be non-zero. In digital and non-digital applications, the sensitivity of nanowire conduction to trapping and de-trapping of charges at interface states can be considered as well.
  	 	 	 	 
 The FET device (see Figure \ref{fig:device} for a schematic diagram)
 consists of two materials, namely silicon ($\DSi$) in the channel and
 source and drain regions and silicon dioxide ($\Dox$) as the
 insulator. The purpose of the insulator is to suppress direct charge
 flow from the gate into the channel and vice versa. The permittivities
 of the materials are $A_\text{Si} = 11.7 A_0$ and $A_\text{ox} = 3.9
 A_0$, where the vacuum permittivity (dielectric constant) is
 $\unit{8.85 \cdot 10^{-12}}{F m^{-1}}$. Moreover, the gates have a
 length of $\unit{30}{nm}$ and are separated from the silicon channel
 by a $\unit{2}{nm}$ thick oxide layer. The channel width is
 $W=\unit{15}{nm}$ and it is connected to the heavily n-type doped
 source and drain regions of length $L_\text{SD} := \unit{10}{nm}$ in each
 region.
 
 Regarding the boundary conditions of the model equations, we apply Dirichlet
 boundary conditions at the gates
 ($V_\text{g}=\unit{0.2}{V}$) and,  the source-to-drain voltage is
 $V_\text{SD}=\unit{0.1}{V}$. Also, the thermal voltage is $U_T=\unit{0.026}{V}$.
  The contacts are illustrated in Figure
 \ref{fig:device}. For the rest of the transistor, we apply zero Neumann boundary
 conditions. In the highly doped source and drain regions, the
 dopant atoms are randomly distributed and indicated in the figure by
 blue circles. Therefore, random-dopant effects are included due to the
 random position of the dopant atoms.
  	 	 	 	 	   
 In the common models, the doping concentration is modeled as a
 macroscopic, deterministic quantity which averages out any microscopic
 non-uniformities due to the random placement and random number of
 dopants. For large devices with a large number of dopants, this
 continuum model is physically reasonable, since the electrostatic
 potential appears spatially homogeneous and is sufficiently well
 described by the averaged charge density. However, in nanoscale
 transistors, the randomly distributed dopant atoms lead to inevitable variations between the billions of transistors in an integrated
 circuit. 
 
  	 	 	 	 	  
 As mentioned already, the main source of device variation is the
 random motion of impurity atoms during the fabrication procedure of
 implantation and annealing. In order to model the stochastic
 coefficients in the model equations, each dopant is modeled as a
 Gaussian distribution such that the doping concentration at point~$x$
 is given by \cite{chen2010modeling}
 \begin{equation}
   \label{dirac}
   C(x,\omega) := \sum_j
   \frac{C_j}{\left(2\pi\sigma^2\right)^{3/2}}
   \exp\left( - \frac{(x-x_j(\omega))^2}{2\sigma^2} \right),
 \end{equation}
 where $x_j$ and $C_j$ are the position and the charge of the $j$-th
 dopant, respectively.  In the source and drain, to determine the position of random dopant, two random points (according to the two-dimensional problem) are used to translate it. For instance, the random variable $\omega=(\frac{1}{2},\frac{1}{2})$ transforms the dopants to the center of a region ($x_j(\omega)$). Here, we assume that both regions have same equal number of dopants.
 Also, $\sigma:=\unit{0.35}{nm}$ corresponds to
 the extent of the electrostatic influence, and the results are not
 significantly sensitive to the value of~$\sigma$. Finally, the source and drain
 regions contain n-type dopants corresponding to a continuous doping
 concentration of $\unit{1\cdot10^{19}}{cm^{-3}}$ (heavily doped) and
 the doping concentration of the channel is
 $\unit{1\cdot10^{16}}{cm^{-3}}$.
 
 
 \begin{figure}[ht!]
   \centering
   \subfloat{\includegraphics[width=0.5\linewidth]{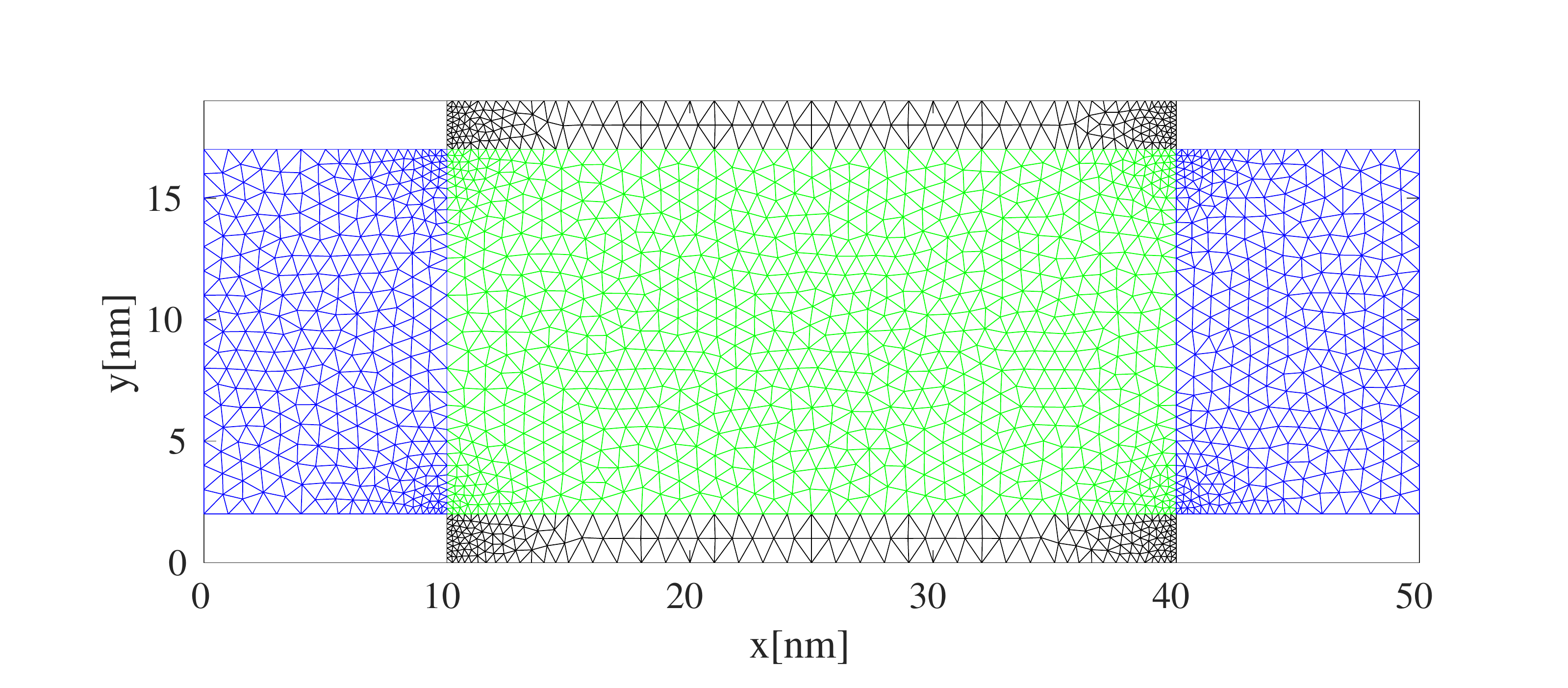}}
   \hfill    
   \subfloat{\includegraphics[width=0.5\linewidth]{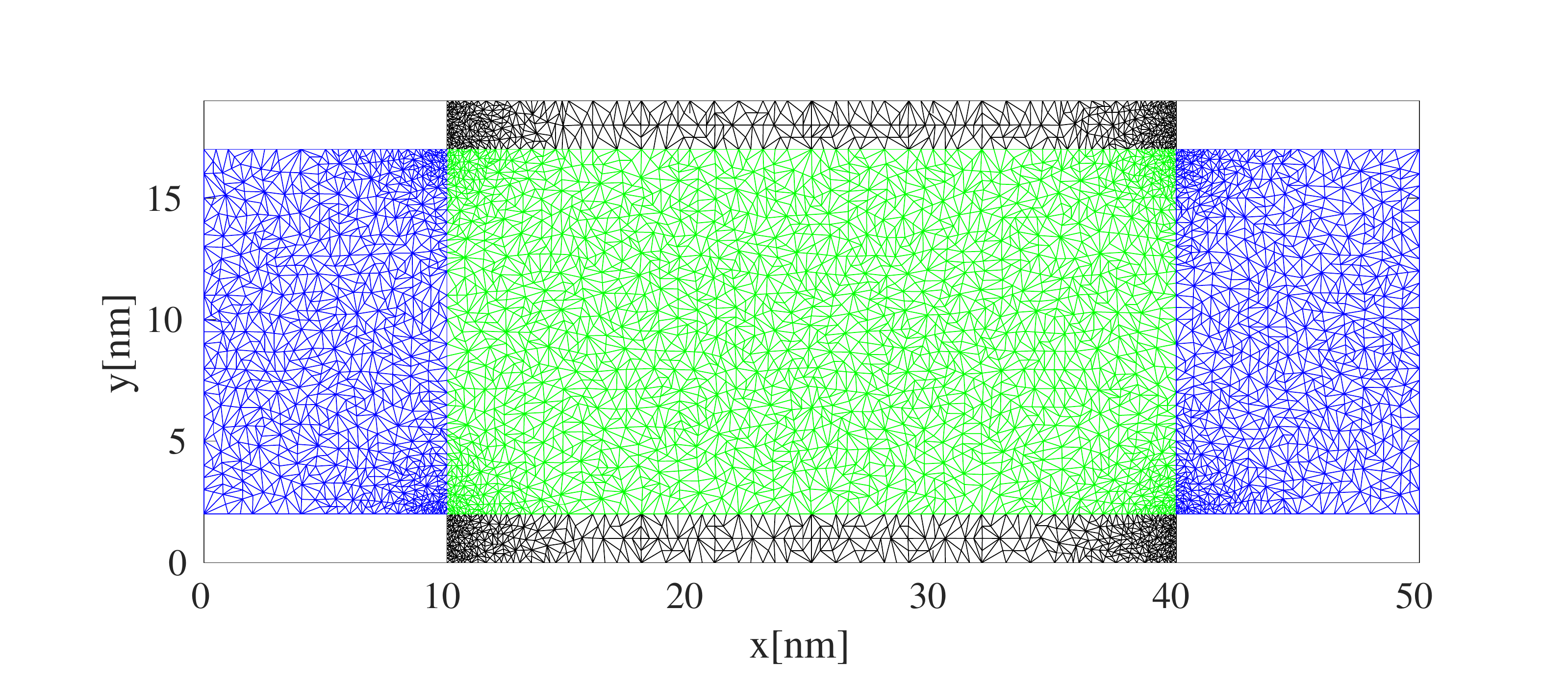}}
   \caption{Initial mesh (left) and uniformly refined mesh (right) for a DG-MOSFET.}
   \label{fig:refine}
 \end{figure}
 
 In the following, we strive to draw a comparison between adaptive
 and uniform MLMC-FE methods.

   Uniform and refined  meshes corresponding to the device
 shown in Figure~\ref{fig:device} are depicted in
 Figure~\ref{fig:refine}.  In adaptive refinement, we use the marking
 strategy introduced in \cite{dorfler1996convergent}. Here, for each
 element $T\in \mathcal{T}_\ell$, the local refinement
 indicator $\eta_T$   satisfies
 \begin{equation}
   \sum_{T \in \mathcal{M}}\eta_T^2\geq\theta\eta^2,
   \label{mark}
 \end{equation}
 where $\mathcal{M}$ is the set of elements marked for refinement and the associated error estimator is defined as
 \begin{equation*}
   \eta := \left(\sum_{T\in\mathcal{T}}\eta_T^2\right)^{1/2}.
 \end{equation*}	 	 	 	 	  
 In other words, we refine the smallest subset of elements whose corresponding error indicators in sum exceed the threshold $\theta \eta^2$.

 The adaptive algorithm for the boundary-value problem is shown
 Algorithm~1. In the multilevel setting, the mesh is
 refined as long as $\mathcal{E}^2_L$ is greater than or equal to
 $\varepsilon^2/2$ and the number of samples are obtained according to the optimization problem (\ref{op:MLMC}). Also, the same number and positions of random
 variables are used on all levels $\ell \in \{1,\ldots,L\}$. In the
 numerical example, we set $\red{\theta:=0.6}$ and the initial mesh
 $\mathcal{T}_0$ and its uniform refinement are depicted in Figure~\ref{fig:refine}.
 
 \begin{algorithm}[ht!]
   \label{algorithm}
   \textbf{Initialization ($\ell=0$):} 
   \vspace{0.1cm}
   
   Initial mesh $\mathcal{T}_0$.
   
   \vspace{0.1cm}
   
   \While{$\mathcal{E}_\ell^2> \epsilon^2/2$}{
     \vspace{0.2cm}
     ~~~\textbf{for} $i=1,\ldots,M_\ell$
     
     \vspace{0.2cm}
     \qquad ~~~  (i)~Solve the boundary-value problem \eqref{modeleqn} to find $V_{h_\ell}^{(i)}$, $v_{h_\ell}^{(i)}$ and $u_{h_\ell}^{(i)}$ according to $M_\ell$.\\
      \qquad ~~~  (ii)~Compute \red{the error indicator}  by
        \eqref{post} for the $i$th sample on
      	all elements.
     
     \vspace{0.1cm}
     ~~~ \textbf{end}
       
       
     \vspace{0.1cm}
       
    ~~~Compute the a posteriori error estimator for the expected
      values of the solutions according to \eqref{coro21}.
       
     \vspace{0.1cm}
       
      ~~~Determine the triangles to be refined  using the
     marking strategy \eqref{mark}.
     
      ~~~$\mathcal{T}_{\ell+1} := \text{refine } (\mathcal{T}_{\ell}, \mathcal{M}_\ell)$ ~where~$\mathcal{M}_\ell$~determined~by~ \eqref{mark}.
 
     \vspace{0.1cm}
     
      ~~~$\ell:=\ell+1$.
 
     \vspace{0.1cm}
     
      ~~~Estimate the discretization error $(\mathcal{E}_\ell)$ according to the refined meshes.
   }
   \caption{The adaptive MLMC-FE strategy for the coupled system of
     equations \eqref{modeleqn}.}
 \end{algorithm}
 
 The adaptively refined meshes for $\ell \in \{1,\ldots,6\}$ for the
 coupled system of equations are shown in Figures~\ref{fig:mesh12},
 \ref{fig:mesh34}, and~\ref{fig:mesh54}. As shown, most of the meshes
 have been refined due to the randomness in the source and drain areas.
 Similarly, the interface condition between the insulator and the
 channel~$\Gamma$ gives rise to more refinements, where less mesh
 refinement occurred in the channel (green triangles). The
 corresponding degree of freedom for the Poisson
 ($\mathcal{N}_{\mathcal{P}_\ell}$) and drift-diffusion
 ($\mathcal{N}_{\mathcal{DD}_\ell}$) equations are summarized in
 Table~\ref{table:NP} and Table~\ref{table:DD}, respectively. We
 compare the obtained degrees of freedom for adaptive and uniform
 refinement, where the initial mesh is the same in both cases.
 
 \begin{figure}[ht!]
   \centering
   \subfloat{\includegraphics[width=0.5\linewidth]{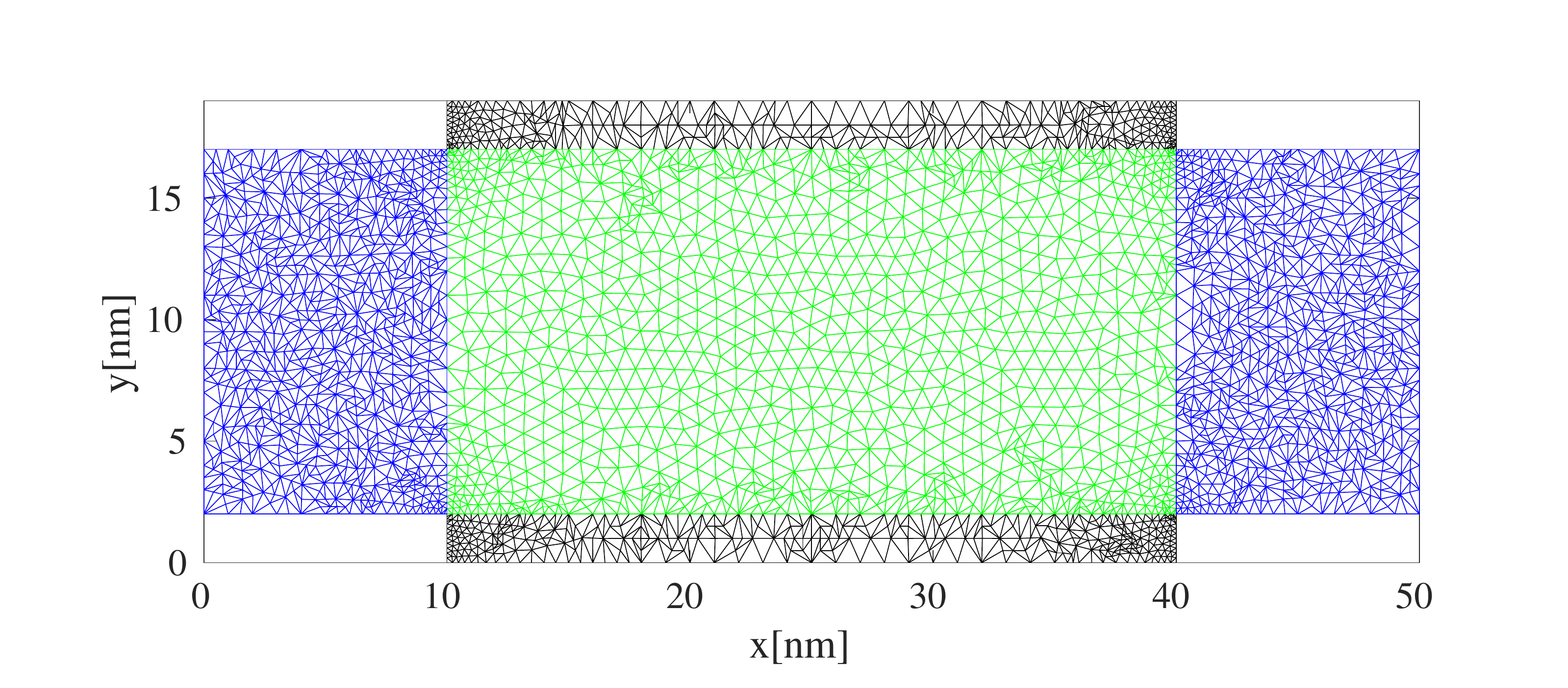}}
   \hfill    
   \subfloat{\includegraphics[width=0.5\linewidth]{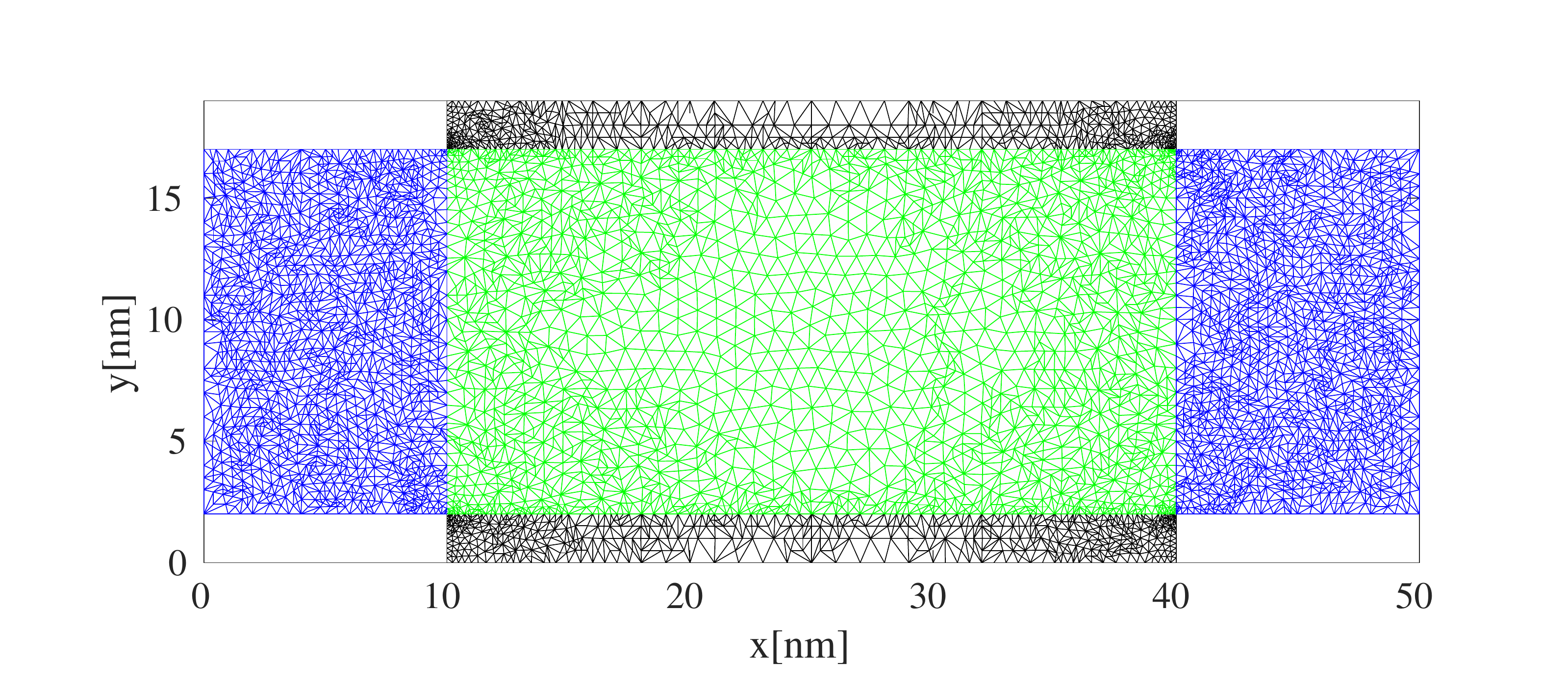}}
   \caption{Adaptive mesh refinement for a DG-MOSFET with random
     dopants at $\ell=1$ (left) and $\ell=2$ (right).}
   \label{fig:mesh12}
 \end{figure}
 
 \begin{figure}[ht!]
   \centering
   \subfloat{\includegraphics[width=0.5\linewidth]{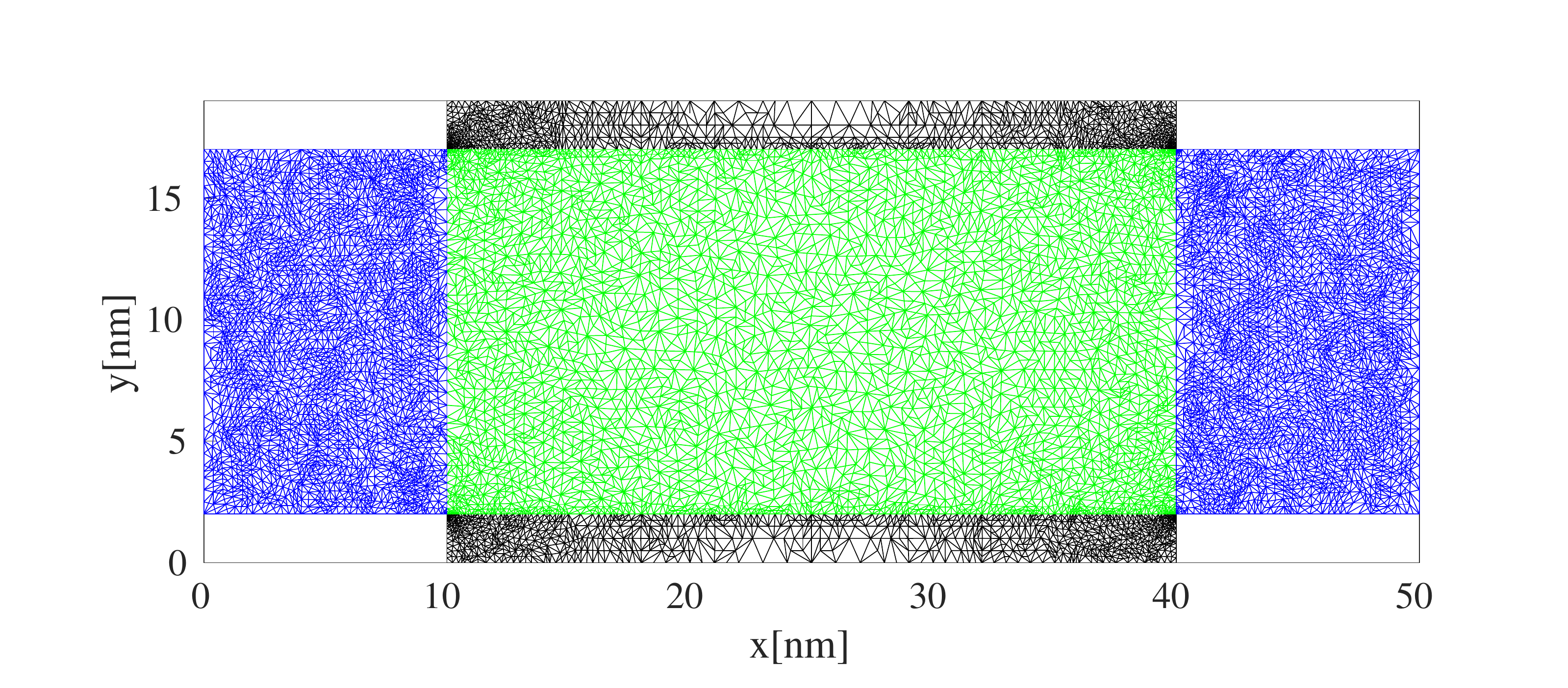}}
   \hfill    
   \subfloat{\includegraphics[width=0.5\linewidth]{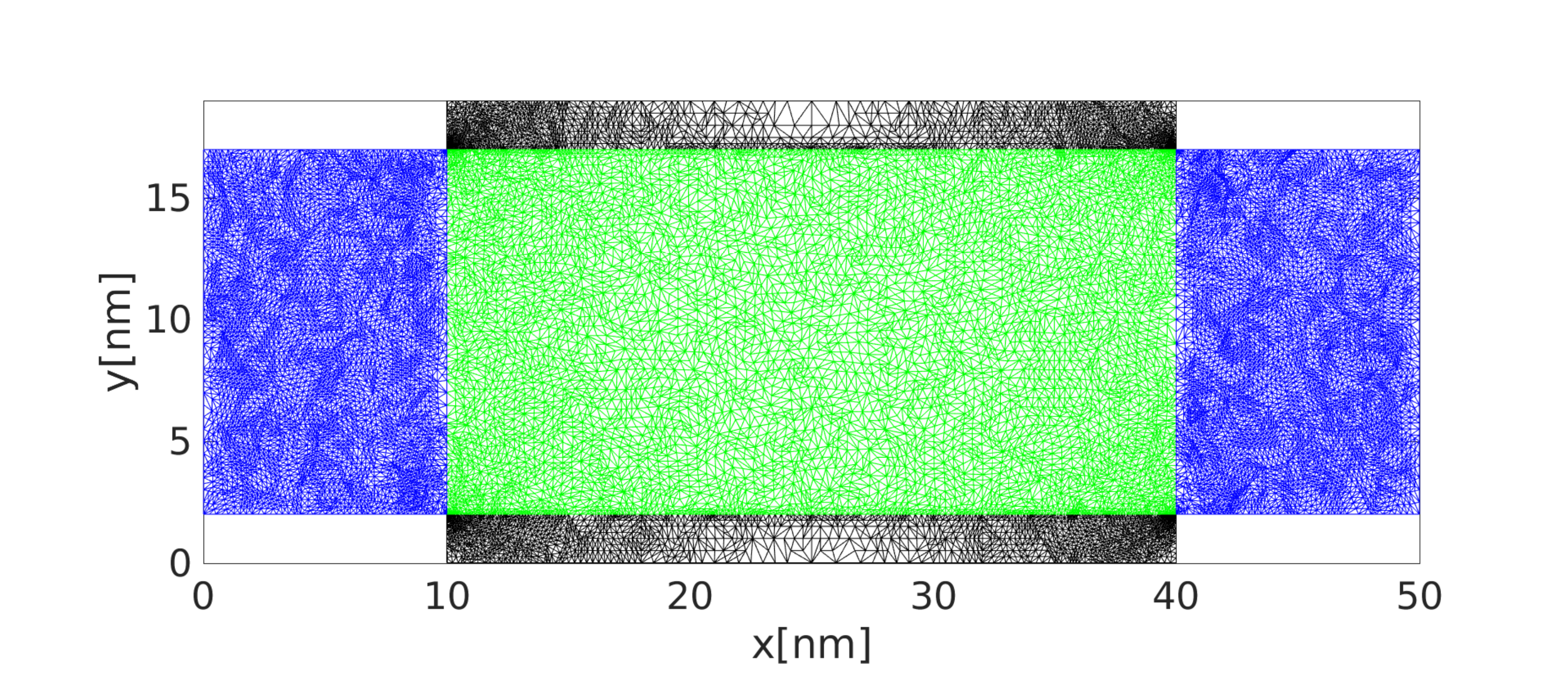}}
   \caption{Adaptive mesh refinement for a DG-MOSFET with random
     dopants at $\ell=3$ (left) and $\ell=4$ (right).}
   \label{fig:mesh34}
 \end{figure}
 
 \begin{figure}[ht!]
   \centering
   \subfloat{\includegraphics[width=0.5\linewidth]{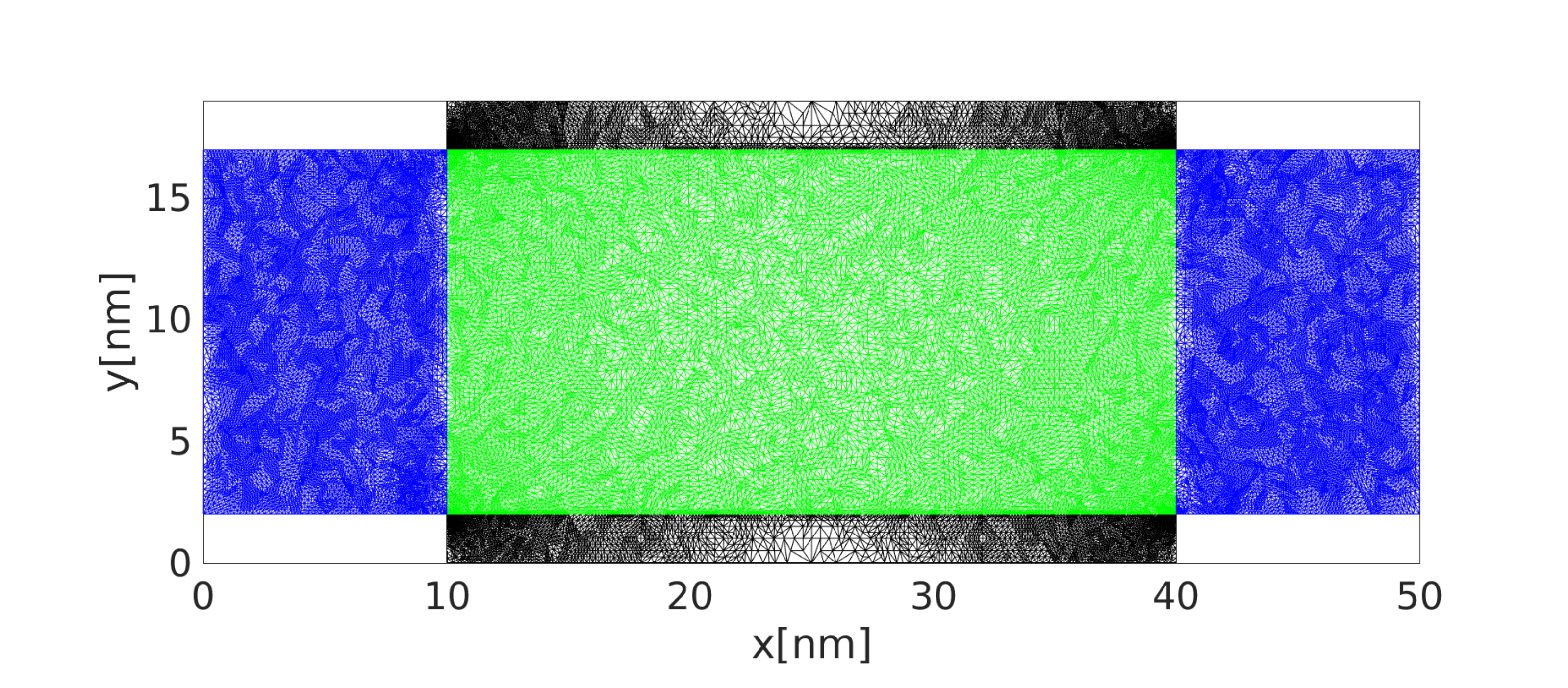}}
   \hfill    
   \subfloat{\includegraphics[width=0.5\linewidth]{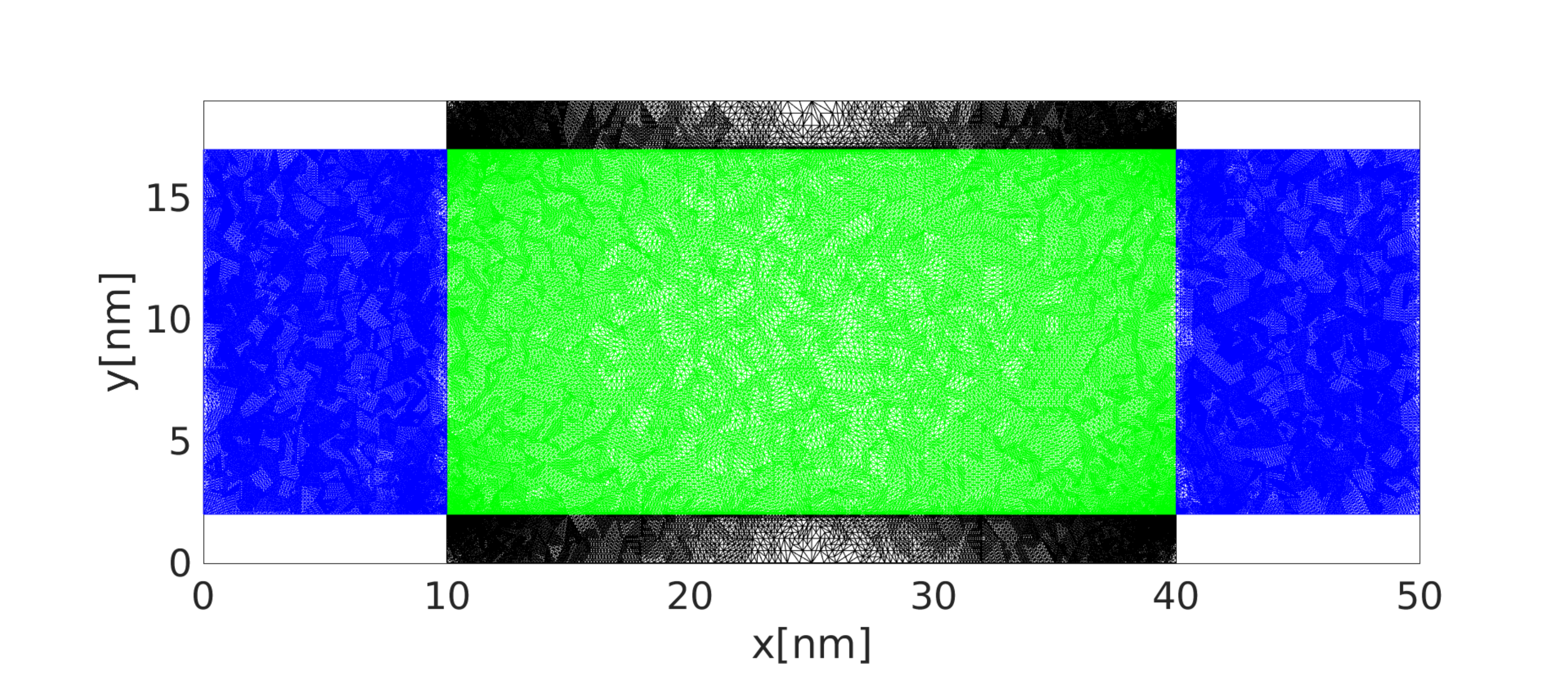}}
   \caption{Adaptive mesh refinement of a DG-MOSFET with random dopants
     at $\ell=5$ (left) and $\ell=6$ (right).}
   \label{fig:mesh54}
 \end{figure}
 
 \begin{table}[ht!]
   \centering
   \begin{tabular}{rrrrrrrr}
     $\ell$  &0  & 1 & 2 & 3 & 4 & 5 & 6 \\
     \hline
     uniform & 1\,685 & 4\,749  & 11\,626 & 27\,084& 60\,569& 131\,819 & 280\,264 \\
     adaptive & 1\,685 & 2\,839   & 5\,531 & 10\,876 & 22\,128 & 45\,885 & 98\,123
   \end{tabular}
   \caption{Degrees of freedom $\mathcal{N}_{\mathcal{P}_\ell}$ for
     different levels comparing uniform MLMC-FE and adaptive MLMC-FE
     methods.}
   \label{table:NP}
 \end{table} 
 
 \begin{table}[ht!]
   \centering
   \begin{tabular}{rrrrrrrr}
     $\ell$  &0  & 1 & 2 & 3 & 4 & 5 & 6 \\
     \hline
     uniform & 755 & 2\,141  & 5\,293 & 12\,372& 27\,743& 60\,509 & 128\,833 \\
     adaptive & 755 & 898   & 1\,640 & 3\,166 & 6\,614 & 13\,553 & 29\,305
   \end{tabular}
   \caption{Degrees of freedom $\mathcal{N}_{\mathcal{DD}_\ell}$ for
     different levels comparing uniform MLMC-FE and adaptive MLMC-FE
     methods.}
   \label{table:DD}
 \end{table}	
 \begin{figure}[ht!]
	\centering
	\includegraphics[width=10.5cm, height=6.0cm]{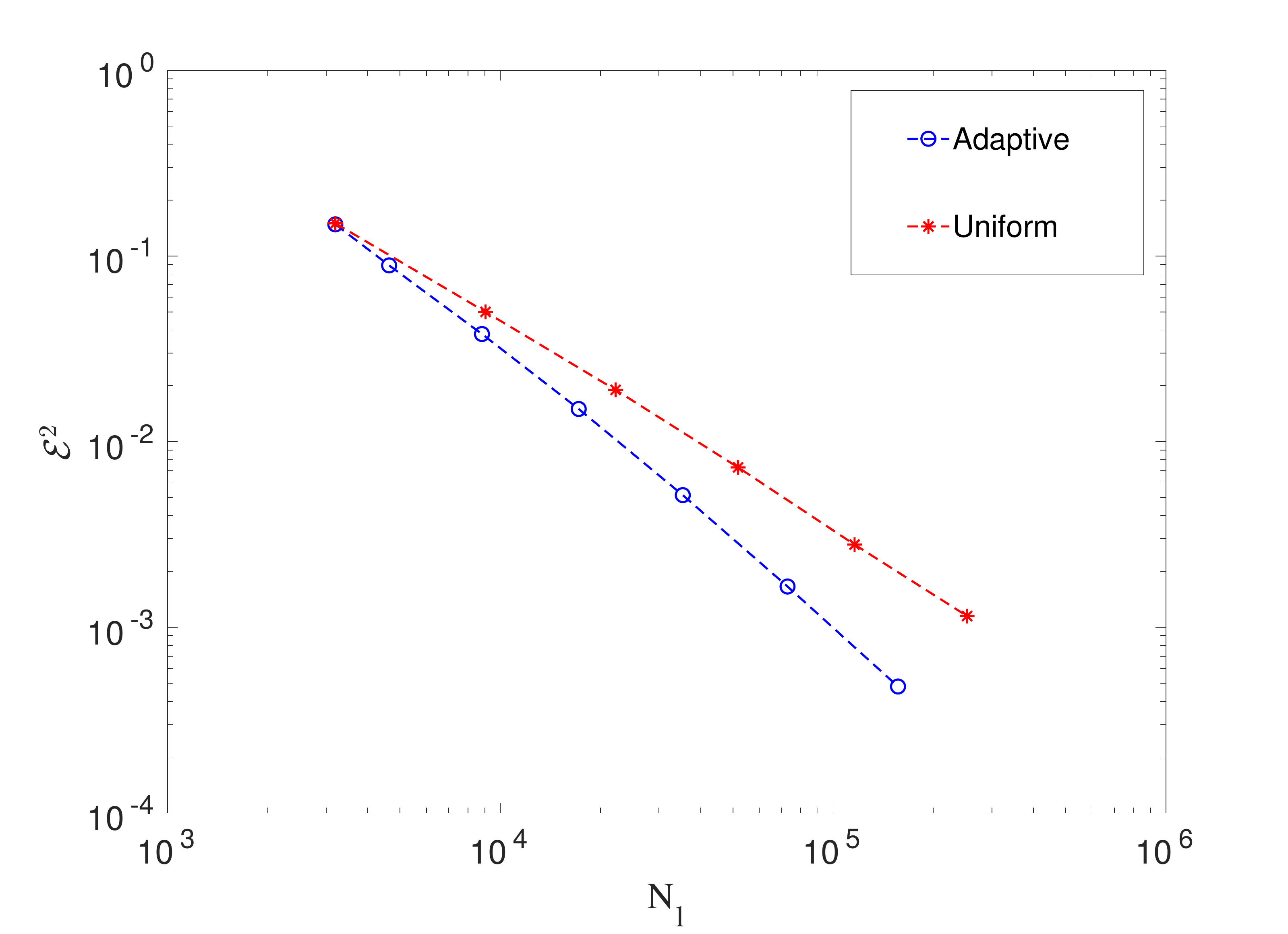}
	\caption{The discretization error and error estimator ($\eta$) of the drift-diffusion-Poisson
		system \eqref{eq:l2error} as a function of the degrees of
		freedom. Here, we consider uniform mesh-refinement as well as the adaptive strategy.}
	\label{fig:dis}
\end{figure}
 \begin{figure}[ht!]
	\centering
	\includegraphics[width=10.5cm, height=6.0cm]{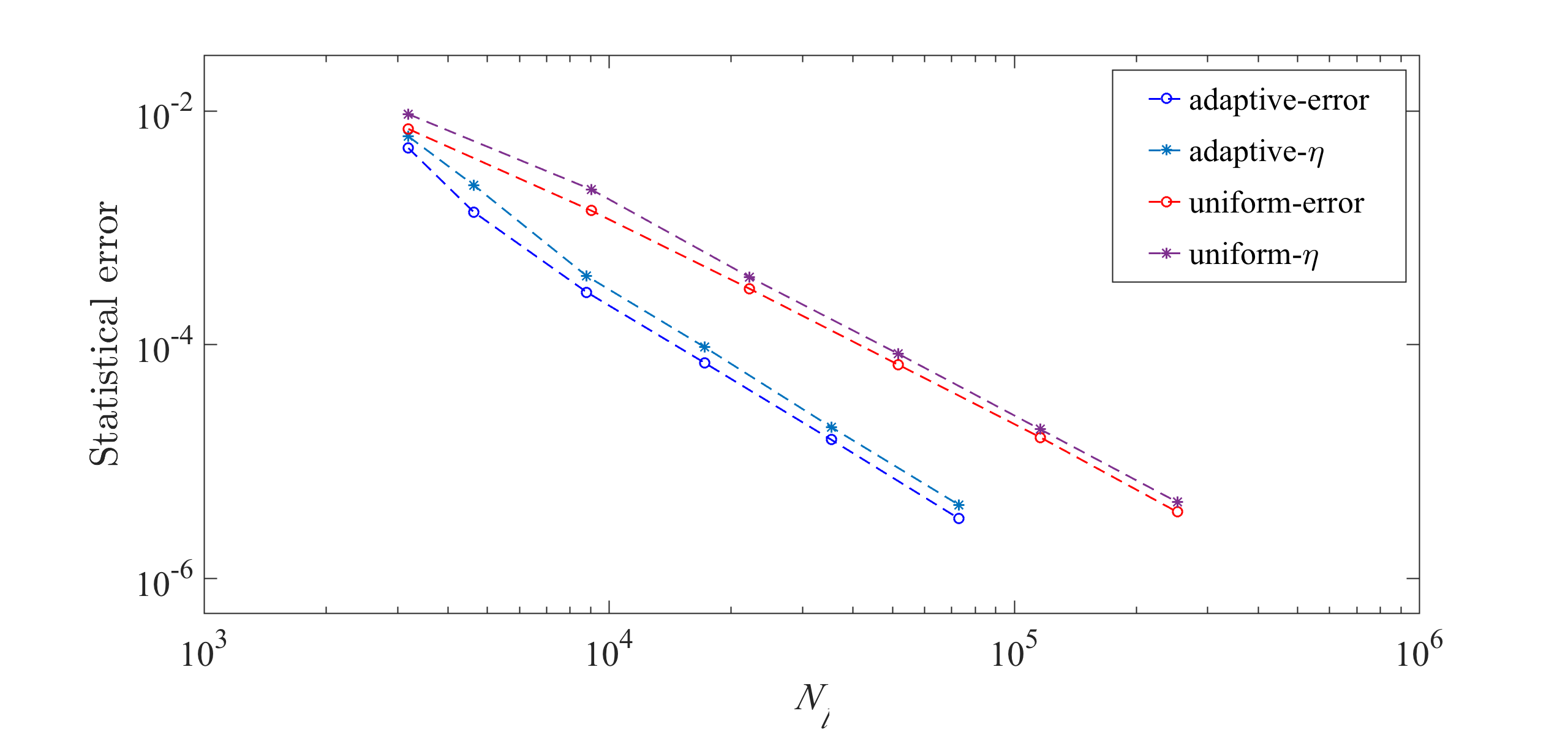}
	\caption{The statistical error (using error and error estimator ($\eta$)) of the drift-diffusion-Poisson system
		\eqref{eq:variance} as a function of the degrees of freedom. Here, we consider uniform mesh-refinement as well as the adaptive strategy.}
	\label{fig:variance}
\end{figure}
 
 In the next step, we compare the discretization error of uniform and
 adaptive refinement \red{with 100 samples}. Figure~\ref{fig:dis} indicates that the adaptive
 refinement reduces the \red{discretization} error, \red{i.e., $\mathcal{E}_\ell$} \red{(with respect to a reference solution),} and obtains a convergence rate of
 $\alpha=1.46$. However, the uniform refinement leads to a smaller
 convergence rate of $\alpha=1.11$. Obviously, using adaptive refinement enables us to obtain a rate closer to the optimal convergence rate given in Corollary \ref{cor1}. Here, we also show the error estimator ($\eta$) for adaptive and uniform mesh-refinement to confirm its efficiency and reliability.
 
 Moreover, we compare the statistical error of both multilevel methods.
 Figure~\ref{fig:variance} illustrates the decay of the variance for
 different degrees of freedom. The results show that similar to the
 discretization error, the variance in the adaptive approach is reduced
 faster ($\beta=2.27$) compared to uniform refinement ($\beta=1.73$).
 Again, the efficiency of the adaptive method is shown by the numerical
 results indicating error and error estimator for adaptive and uniform refinements.
 Also, $C_0=0.041$ is obtained as the variance of level
 $\ell=0$.

 \begin{table}[t!]
 	\centering
 	\begin{tabular}{rrrrrrrr}
 		$\varepsilon$  &$M_0$  & $M_1$ & $M_2$ & $M_3$ & $M_4$& $M_5$      \\
 		\hline
 		0.080  & 17 & 5  & -- & --&--&--&  \\
 		0.040  & 90 & 26  & 8 & --&--&--&   \\
 		0.020  & 418  & 121  & 35 & 11& --&--&  \\
 		0.010  & 1\,671 & 481  & 137 & 42& --&--&    \\
 		0.007  & 3\,759 & 1\,080  & 308 & 94 & 30&--&    \\
 		0.005  & 7\,366 & 2\,117  & 604 & 183 & 58&--&    \\
 		0.002  & 5\,3764 & 17\,113& 4\,163 & 1\,012 & 573 & 57&    \\
 	\end{tabular}
 	\caption{The optimal number of samples for the uniform MLMC-FE method.}
 	\label{table:5}
 \end{table}
 
 \begin{table}[t!]
 	\centering
 	\begin{tabular}{rrrrrrrrrr}
 		$\varepsilon$  &$M_0$  & $M_1$ & $M_2$ & $M_3$ & $M_4$& $M_5$& $M_6$&   \\
 		\hline
 		0.080  & 18 & 6  & 3 & --&--&--&--&  \\
 		0.040  & 78 & 26  & 10 & 4&--&--&--&  \\
 		0.020  & 334  & 111  & 44 & 14& 5&--&--&   \\
 		0.010  & 1\,369 & 472  & 186 & 41& 20& --& --&   \\
 		0.007  & 2\,856 & 949  & 372 & 119 & 42& 14& --&   \\
 		0.005  & 5\,597 & 1\,859  & 729 & 233 & 81& 27& --&   \\
 		0.002  & 36\,068 & 11\,979  & 4\,692 & 1\,502 & 520& 170 & 54&   \\
 	\end{tabular}
 	\caption{The optimal number of samples for the adaptive MLMC-FE method.}
 	\label{table:4}
 \end{table}
 In order to estimate the optimal computational complexity, we solve
 the optimization problem \eqref{op:MLMC}. An interior point method can
 be used to solve the global optimization problem
 \cite{taghizadeh2017optimal}, where the results are the optimal number
 of samples. For different tolerances~$\varepsilon$, the optimal values
 are summarized in Table~\ref{table:5} and Table~\ref{table:4} for
 uniform and adaptive refinements, respectively. In multilevel methods,
 most of the work is performed on coarse levels. The main reason is the
 reduction of variance on the finer grids.
 
 The computational work $\sum_{\ell=0}^{L} M_\ell\mathcal{N}_\ell$ for
 both refinement methods are depicted in Figure~\ref{fig:work}. We here
 observe a significantly better efficiency of the adaptive model compared
 with the uniform approach. As depicted in the figure, the
 computational cost asymptotically behaves like $O(\varepsilon^{-2})$
 for both multilevel techniques, which agrees with
 \cite{charrier2013finite}. A more interesting computational result achieved regarding the CPU time. As Figure ~\ref{fig:work} shows
  since in the adaptive approach (compared to the uniform refinement) fewer degrees of freedom are needed, in different levels, lower computational time is used (to obtain same error tolerance). The difference between two computational times is more pronounced for lower prescribing errors that indicates the adaptive technique efficiency.

 \begin{figure}[ht!]
	\centering
	\subfloat{\includegraphics[width=8cm, height=5cm]{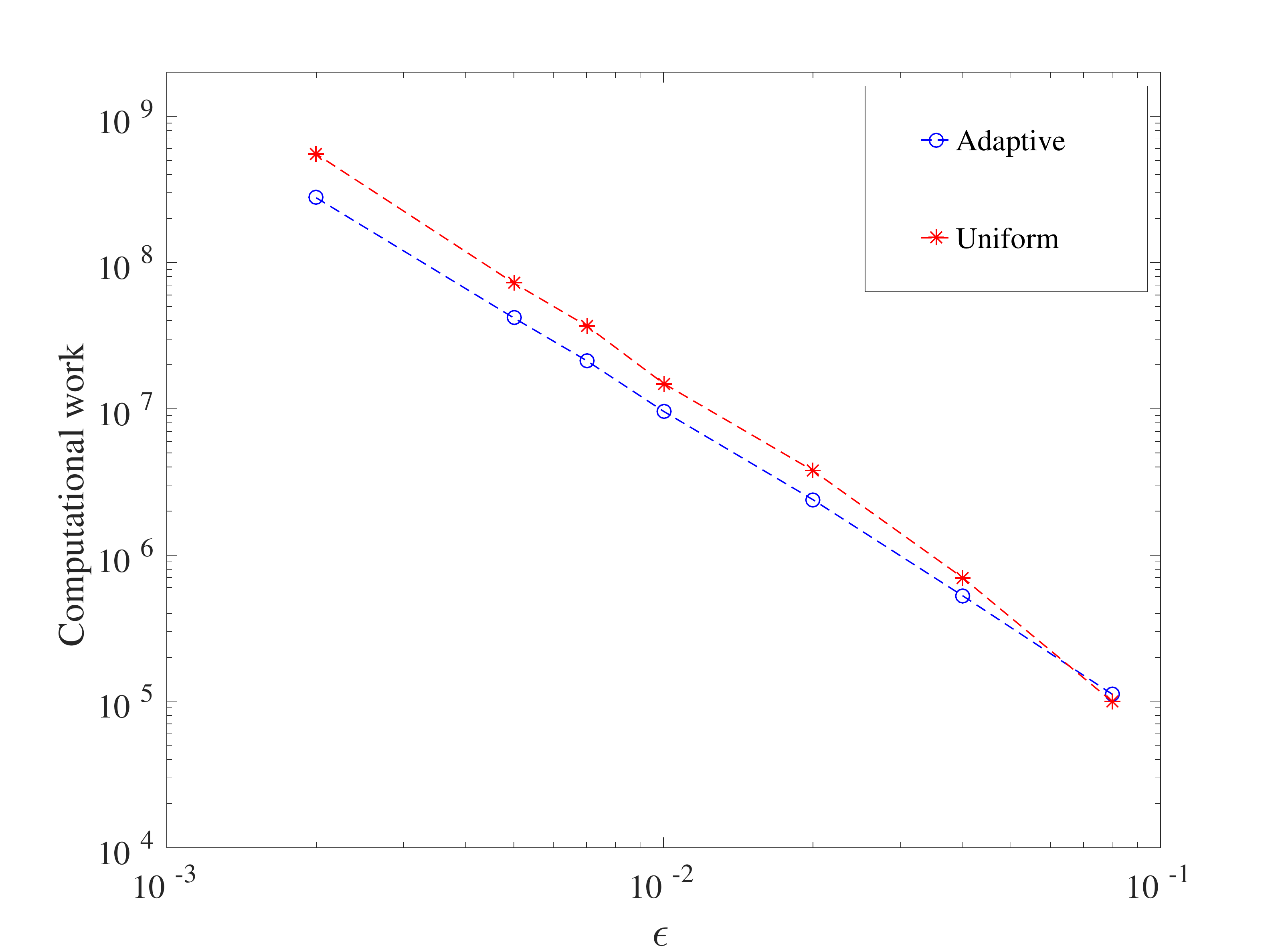}}
	\hfill    
	\subfloat{\includegraphics[width=8cm, height=5cm]{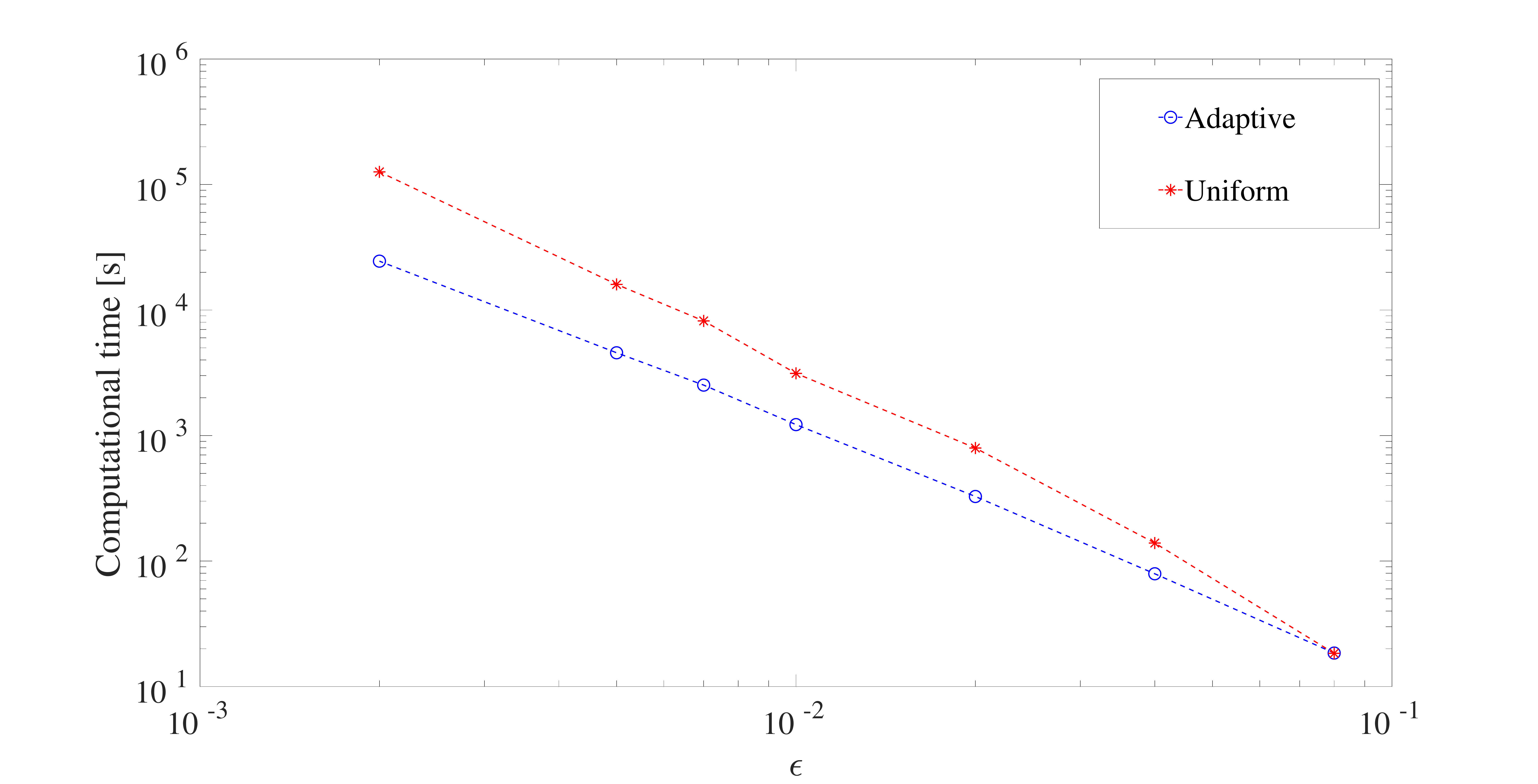}}
	\caption{The computational cost (left) and computational cost (right) of uniform MLMC-FE and adaptive
		MLMC-FE methods for different prescribed tolerances.}
	\label{fig:work}
\end{figure}	 	 	 	
 
 \section{Conclusions}
 
 We have presented an adaptive MLMC-FE method for the numerical
 solution of the stochastic drift-diffusion-Poisson system. First, we
 proved an a-priori error estimate for the coupled system of equations
 with non-zero recombination rate. The error estimate points out how
 fast the error decreases as the mesh size decreases and can be
 considered as a useful measure of the efficiency of a given
 finite-element method. Also, using the stochastic numerical example,
 we estimated the convergence rate of the discretization error.
 
 Secondly, a practically useful a-posteriori error indicator to bound
 the discretization error for the coupled system of equations was
 derived. From a computational point of view, the error estimator is
 inexpensive to estimate and guarantees the bounds on the error on all
 points of the geometry. The error indicator was used to design an
 adaptive refinement strategy to refine the mesh, where all
 coefficients in the system of equations can be random. In
   future works, the error analysis performed here can be implemented
   for extended systems of equations accounting for quantum effects at
   first order in the direction perpendicular to the gate (e.g.,
   density gradient corrections to the charge term in the Poisson
   equation).
 
 Regarding numerical examples, we implemented this adaptive MLMC-FE
 method to quantify noise and variations in nanoscale transistors as a
 real-world example. To this end, we defined a strategy to refine the
 meshes in the stochastic setting. The new technique was compared to
 the multilevel method with uniform refinement as a useful benchmark.
 Better convergence of the discretization error and better decay of
 variance were observed indicating the efficiency of the new approach.
 Finally, we employed an optimization problem to minimize the
 computational complexity. The optimal numbers of samples are obtained
 as the solution of the global optimization problem. The results
 indicate that in addition to a better control of error, a noticeable
 reduction of the computational work/time are achieved by the adaptive
 method.
 
 \section{Acknowledgments}
 
 The first and the last authors acknowledge support by FWF (Austrian
 Science Fund) START project no.\ Y660 \textit{PDE Models for
   Nanotechnology}. The second author also acknowledges support by FWE  project no.\ P28367-N35. The authors appreciate useful comments by
 Markus Melenk (TU Wien).  The authors also acknowledge the helpful comments
 	by three anonymous reviewers.
  	
 \bibliographystyle{elsarticle-num}
 \bibliography{MLMC2}

\end{document}